\documentclass[reqno]{amsart}
\usepackage{color}
\usepackage{amssymb,amsmath,amsthm,amstext,amsfonts}
\usepackage[dvips]{graphicx}
\usepackage{psfrag}
\usepackage{url}

\makeatletter \@addtoreset{equation}{section} \makeatother

\renewcommand\thetable{\thesection.\@arabic\c@table}

\theoremstyle{plain}

\newtheorem{theorem}{Theorem}[section]
\newtheorem{proposition}{Proposition}[section]
\newtheorem{lemma}{Lemma}[section]
\newtheorem{corollary}{Corollary}[section]
\newtheorem{definition}{Definition}[section]
\newtheorem{remark}{Remark}[section]

\begin{document}

	\title{Induced Topological Pressure for Dynamical Systems}

	\author{Wenhui Ma}
	\address{School of Mathematical Sciences, Soochow University\\
	 Suzhou 215006, Jiangsu, P.R. China}
	\address{Center for Dynamical Systems and Differential Equation, Soochow University\\
	Suzhou 215006, Jiangsu, P.R. China}
	\email{20234207032@stu.suda.edu.cn}

	\author{Yun Zhao}
	\address{School of Mathematical Sciences, Soochow University\\
		Suzhou 215006, Jiangsu, P.R. China}
	\address{Center for Dynamical Systems and Differential Equations, Soochow University\\
		Suzhou 215006, Jiangsu, P.R. China}
	\email{zhaoyun@suda.edu.cn}
	
\author{Hanjing Zhu}
	\address{School of Mathematical Sciences, Soochow University\\
		Suzhou 215006, Jiangsu, P.R. China}
	\address{Center for Dynamical Systems and Differential Equations, Soochow University\\
		Suzhou 215006, Jiangsu, P.R. China}
	\email{20234207026@stu.suda.edu.cn}

	\thanks{This work is partially supported by National Key R\&D Program of China (2022YFA1005802) and NSFC (12271386).}

	
	\begin{abstract}
This paper is devoted to the study of  induced topological pressure, including both  classical and nonlinear cases. For the classical induced topological pressure, we investigate equilibrium states, subdifferential and freezing states, while also discussing some basic properties of it. Additionally, the high dimensional nonlinear induced  topological pressure is introduced, and the corresponding variational principle is established.
	\end{abstract}

	

	\maketitle

	\section{Introduction }\label{intro}

	In 1958, Kolmogorov and Sinai \cite{Kolmogorov1958} introduced the notion of measure theoretic entropy in dynamical systems, which is an important invariant quantity under metric isomorphisms. In 1965, Adler, Konheim and McAndrew \cite{adler1965} proposed the concept of topological entropy.
Measure theoretic entropy primarily describes the complexity of a dynamical system under a given invariant measure, while topological entropy does not depend on a specific measure and only considers the topological properties of the system, such as the distribution of orbits and the number of periodic points. The variational principle of entropy further points out the intrinsic relationship between these two quantities, i.e.,  when the measure changes within the set of invariant measures, the supremum of measure theoretic entropy equals the topological entropy.
In 1973, Bowen \cite{bowen1973topological} extended the concept of topological entropy to non-compact invariant sets, proposed topological entropy with dimension features, and proved the corresponding variational principle.

Topological pressure is a generalization of the topological entropy, which  measures the complexity of system behavior under a given observable function. It plays an important role in the theory of thermodynamic formalism and provides fundamental tools for the study of dimension theory in dynamical systems, see  \cite{barreira2011thermodynamic,pesin1997dimension,rue04} for detailed descriptions. In 1973, Ruelle \cite{ruelie1973statistical}  introduced the specific definition of topological pressure, and obtained a variational principle. Walters \cite{walters1975variational} further generalized it to general continuous maps on a compact metric space. Pesin and Pitskel \cite{pesin1984topological} defined topological pressure on non-compact sets, which is an extension of the topological entropy defined by Bowen \cite{bowen1973topological}, and proved the variational principle under some conditions.
The study of topological pressure is closely related to other fields of the theory of dynamical systems (see \cite{barreira2011thermodynamic,bowen1975equilibrium,katok1995introduction,keller1998equilibrium,parry1990zeta}). In particular, the  topological pressure  plays a fundamental role in statistical mechanics, ergodic theory and dynamical systems (see \cite{pesin1997dimension,katok1995introduction}).

 As the iteration of a dynamical system may involve different time scales, the concepts of induced entropy and induced topological pressure have garnered significant attention.
 In \cite{jaerisch2014induced}, Jaerisch, Kesseb\"{o}hmer and Lamei  introduced the notion of induced topological pressure for Markov shifts over an infinite alphabet, generalizing the result of Savchenko \cite{savchenko1998special}. The key advantage of induced topological pressure lies in the freedom to choose the scaling function \( \psi > 0 \), which finds applications in the theory of large deviations and  fractal geometry. In particular, using induced topological pressure, some thermodynamic quantities can be explicitly characterized by choosing an appropriate scaling function \( \psi \), thus providing more flexible analytical tools.
  Inspired by the work in \cite{jaerisch2014induced}, Xing and Chen \cite{xing2015induced} further defined induced topological pressure for  general topological dynamical systems and established the variational principle.
  Later, they established a relation between two induced topological pressures with a factor map in  \cite{xing2017induced}, pointing  out that the BS dimension (see \cite{barreira2000sets} for details) is a special case of  induced topological pressure. Following  Katok's idea, in \cite{xing2018katok} the authors introduced the notion of induced measure theoretic entropy,  proved its Katok formula, and  obtained a variational principle for the BS dimension via induced measure theoretic entropy. Recently, Rahimi and Ghodrati \cite{rahimi2025abramov} defined  induced topological pressure in random dynamical systems, and obtained the variational principle.

   Based on the research of the Curie-Weiss mean field theory \cite{leplaideur2019generalized}, Buzzi et al \cite{buzzi2020nonlinear} developed the theory of  nonlinear thermodynamical formalism, introducing the statistical mechanics of generalized mean field models into dynamical systems. Under appropriate conditions, they established the variational principle of nonlinear topological pressure and characterized the correspondence between nonlinear and linear equilibrium measures. Later, Barreira and Holanda  proposed the higher-dimensional version of the nonlinear thermodynamical formalism  for discrete dynamical systems and flows in \cite{barreira2022higher} and \cite{barreira2022nonlinear} respectively,  and established the  variational principles .  Recently, Ding and Wang \cite{ding2025some} studied nonlinear topological pressure  using the theory of Carath\'{e}odory structures, introduced several types of nonlinear topological  pressures on a subset, and established several variational principles in different forms for the nonlinear topological pressures. Following the approach described in \cite{fh16}, Yang et al \cite{ycz24} introduced nonlinear weighted topological pressure, and established variational principle for nonlinear weighted topological pressure under  suitable conditions.

  For the induced topological pressure  introduced in \cite{xing2015induced}, this paper provides several equivalent definitions  using open covers, and explores some basic properties, such as monotonicity, continuity, convexity and differentiability of induced topological pressure. In particular, this paper introduces the concepts of freezing states and zero temperature limits of induced topological pressure, and studies the relationship between freezing states and equilibrium states. Additionally, following the approaches described in  \cite{barreira2022higher} and \cite{xing2015induced},
  this paper defines high dimensional  nonlinear induced topological pressure, and establishes  the variational principle for this newly defined thermodynamic quantity.

  This paper is organized as follows: Section \ref{PM} gives some fundamental properties of the induced topological pressure that is introduced by Xing and Chen \cite{xing2015induced}, including equivalent definitions, monotonicity, continuity, convexity and differentiability. Section \ref{ES and FS} studies the basic properties of equilibrium states, subdifferential and  freezing states of induced topological pressure, and establishes the relationship between freezing states and equilibrium states. In Section \ref{NTP}, we first give some basic properties of high dimensional nonlinear topological pressure , which is defined in \cite{barreira2022higher}. Moreover, following the approach as described in \cite{xing2015induced}, we define the high dimensional nonlinear induced topological pressure, and establish the variational principle for it.

	\section{Properties of Induced Topological Pressure}\label{PM}
In this section, we first recall the classical induced topological pressure defined in \cite{xing2015induced}, and then give some properties of it.

Let $(X, f)$ be a topological dynamical system (TDS for short), that is, $f: X \rightarrow X$ is a continuous map on a compact metric space  $(X,d)$, and let
\[
B(x, \varepsilon) = \left\{ y \in X : d(x, y) < \varepsilon \right\} \quad \text{and} \quad \bar{B}(x, \varepsilon) = \left\{ y \in X : d(x, y) \leqslant \varepsilon \right\}.
\]
 Given \( n \in \mathbb{N} \), $\varepsilon>0$ and $x\in X$, denote by \( B_{n}(x, \varepsilon):=\{ y\in X: d_{n}(x, y)<\varepsilon \} \) the Bowen ball of radius $\varepsilon$ centered at $x$ of length $n$, where
$$
d_{n}(x, y):= \max \{ d( f^{i}(x), f^{i}(y) ) : i = 0, 1, \cdots, n-1 \},$$ and let
\( \bar{B}_{n}(x, \varepsilon) \)  denote the closure of $ B_{n}(x, \varepsilon)$.
Let \( Z \subseteq X \) be a non-empty set. A subset \( F_{n} \subset X \) is called an \( (n, \varepsilon) \)-spanning set of \( Z \), if for each \( y \in Z \), there exists \( x \in F_{n} \) such that \( d_{n}(x, y) \leqslant \varepsilon \). A subset \( E_{n} \subset Z \) is called an \( (n, \varepsilon) \)-separated set of \( Z \), if every two distinct points \( x, y \in E_{n} \) implies that \( d_{n}(x, y) > \varepsilon \).

Let $C(X, \mathbb{R})$ denote the Banach space of all continuous functions on $X$,  equipped with the supremum norm $\|\cdot \|$. For each $\varphi\in C(X, \mathbb{R})$, let $S_n\varphi:=\sum_{i=0}^{n-1}\varphi\circ f^i$.
Denote by \( \mathcal{M}(X,f) \) the space of all \( f \)-invariant Borel probability  measures on \( X \), and by \( \mathcal{E}(X,f) \) the set of all \( f \)-ergodic Borel probability  measures on \( X \).

Given \( \varphi, \psi \in C(X, \mathbb{R}) \) with \( \psi > 0 \). For every \( T > 0 \), let
\[
S_{T} = \left\{ n \in \mathbb{N} : \exists x \in X \text{ such that } S_{n} \psi(x) \leqslant T \text{ and } S_{n+1} \psi(x) > T \right\}.
\]
For \( n \in S_{T} \), let
\[
X_{n} = \left\{ x \in X : S_{n} \psi(x) \leqslant T \text{ and } S_{n+1} \psi(x) > T \right\}
\]
and
\[
Q_{\psi, T}(f, \varphi, \varepsilon) = \inf \Big\{ \sum_{n \in S_{T}} \sum_{x \in F_{n}} e^{S_{n} \varphi (x)} : F_{n}\text{ is an } (n, \varepsilon) \text{-spanning set of } X_{n}, n \in S_{T} \Big\}.
\]
 The  following quantity
\[
P_{\psi}(\varphi) = \lim _{\varepsilon \rightarrow 0} \limsup _{T \rightarrow \infty} \frac{1}{T} \log Q_{\psi, T}(f, \varphi, \varepsilon)
\]
is called the \( \psi \)-\emph{induced topological pressure} of $\varphi$ (with respect to \( f \)).

It is easy to see that 	$P_{\psi}(\varphi)=P_{\mathrm{top}}(f, \varphi)$ provided that $\psi(x)\equiv 1$, where $P_{\mathrm{top}}(f, \varphi)$ is the classical topological pressure of $\varphi$ w.r.t.  \( f \), and let $h_{\mathrm{top}}(f):=P_{\mathrm{top}}(f, 0)$ denote the topological entropy of $f$,  see \cite{walters1982introduction} for more details.
	
	\subsection{Equivalent definitions of induced topological pressure}\label{equivalent Top-Pressure}
We first recall the equivalent definition of induced topological pressure defined by separated sets as given in \cite{xing2015induced}.
For every \( T > 0 \), let
\[
P_{\psi,T}(f, \varphi, \varepsilon) = \sup \Big\{ \sum_{n \in S_T} \sum_{x \in E_n}
e^{ S_n \varphi(x) }
: E_n\,\text{is an}\, (n, \varepsilon)\text{-separated set of}\, X_n, n \in S_T \Big\},
\]
one has that
\[
P_{\psi}(\varphi) = \lim _{\varepsilon \rightarrow 0} \limsup_{T \rightarrow \infty} \frac{1}{T} \log P_{\psi, T}(f, \varphi, \varepsilon),
\]
see \cite{xing2015induced} for the detailed proofs.

In the following, we define the induced topological pressure via open covers, and show that it is equivalent to the ones in \cite{xing2015induced}.

Given \( \varphi, \psi \in C(X, \mathbb{R}) \) with \( \psi > 0 \) and an open cover \( \alpha \) of \( X \). For every \( T > 0 \),  let
\[
q_{\psi, T}(f, \varphi, \alpha)= \inf \Big\{\sum_{n \in S_{T}}\sum_{B \in \beta_{n}} \inf _{x \in B} e^{S_{n} \varphi(x)} : \beta_{n} \subset \bigvee_{i=0}^{n-1} T^{-i} \alpha, X_{n} \subset \bigcup_{B \in \beta_{n}}B, n \in S_{T}  \Big\}
\]
and
\[
p_{\psi, T}(f, \varphi, \alpha)= \inf \Big\{\sum_{n \in S_{T}}\sum_{B \in \beta_{n}} \sup _{x \in B} e^{S_{n} \varphi(x)} : \beta_{n} \subset \bigvee_{i=0}^{n-1} T^{-i} \alpha, X_{n} \subset \bigcup_{B \in \beta_{n}}B, n \in S_{T}  \Big\}.
\]
Clearly, \( q_{\psi, T}(f, \varphi, \alpha) \leqslant p_{\psi, T}(f, \varphi, \alpha) \).

\begin{lemma} \label{equivalent definition lm}
Let \( (X, f) \) be a TDS, \( \varphi, \psi \in C(X, \mathbb{R}) \) with \( \psi > 0 \), and \( T > 0 \).
\begin{enumerate}
\item[(i)] If \( \alpha \) is an open cover of \( X \) with Lebesgue number \( \delta \), then
\[
q_{\psi, T}(f, \varphi, \alpha) \leqslant Q_{\psi, T}(f, \varphi, \delta / 2) \leqslant P_{\psi, T}(f, \varphi, \delta / 2).
\]
\item[(ii)] If \( \varepsilon>0 \), and \( \gamma \) is an open cover of \( X \) with \( \operatorname{diam}( \gamma ):=\sup\{|A|: A\in \gamma\} \leqslant \varepsilon \), then
\[
Q_{\psi, T}(f, \varphi, \varepsilon) \leqslant P_{\psi, T}(f, \varphi, \varepsilon) \leqslant p_{\psi, T}(f, \varphi, \gamma).
\]
\end{enumerate}
\end{lemma}
\begin{proof}
For every \( \varepsilon>0 \),  we have that
\( Q_{\psi, T}(f, \varphi, \varepsilon) \leqslant P_{\psi, T}(f, \varphi, \varepsilon) \), since an $(n, \varepsilon)$-separated set of maximum cardinality is  necessarily an $(n, \varepsilon)$-spanning set.

(i) For every \( n \in S_{T} \), if \( F_{n} \) is an \( (n, \delta / 2) \)-spanning set of \( X_{n} \), then
\[
X_{n} \subset \bigcup_{x \in F_{n}} \bigcap_{i=0}^{n-1} f^{-i} \bar{B}\left(f^{i} x , \delta / 2\right).
\]
Since $\delta$ is the Lebesgue number of $\alpha$, for each $x \in F_{n}$, there exists an open subset $B_{x}\in  \bigvee_{i=0}^{n-1} T^{-i} \alpha$ such that $\bar{B}_{n}\left( x , \delta / 2\right) \subset B_{x}$. Therefore, one can find a sub-collection $\beta_{n}=\{B_{x}:x \in F_{n}\}$ of $\bigvee_{i=0}^{n-1} T^{-i} \alpha$ such that  $X_{n} \subset \bigcup_{B \in \beta_{n}}B$, which implies that
\[
\sum_{B \in \beta_{n}} \inf _{x \in B} e^{S_{n} \varphi(x)} \leqslant \sum_{x \in F_{n}} e^{S_{n} \varphi(x) }.
\]
Hence, we have that
\[
\sum_{n \in S_{T}}\sum_{B \in \beta_{n}} \inf _{x \in B} e^{S_{n} \varphi(x)} \leqslant \sum_{n \in S_{T}}\sum_{x \in F_{n}} e^{S_{n} \varphi(x) },
\]
and consequently
\[
q_{\psi, T}(f, \varphi, \alpha) \leqslant \sum_{n \in S_{T}}\sum_{x \in F_{n}} e^{S_{n} \varphi(x) }.
\]
Therefore, \( q_{\psi, T}(f, \varphi, \alpha) \leqslant Q_{\psi, T}(f, \varphi, \delta / 2) \).

(ii) For every \( n \in S_{T} \), take an $(n, \varepsilon)$-separated set $E_{n}$ of $X_{n}$. Since \( \operatorname{diam}( \gamma ) \leqslant \varepsilon \), no element of \( \bigvee_{i=0}^{n-1} f^{-i} \gamma \) can contain two points from \( E_{n} \). Let $\beta_{n}$ be a sub-collection of $\bigvee_{i=0}^{n-1} T^{-i} \gamma$ with $X_{n} \subset \bigcup_{B \in \beta_{n}}B$, then we have that
\[
\sum_{x \in E_{n}} e^{S_{n} \varphi(x)} \leqslant \sum_{B \in \beta_{n}} \sup _{x \in B} e^{S_{n} \varphi(x)}.
\]
Therefore,
\[
\sum_{n \in S_{T}}\sum_{x \in E_{n}} e^{S_{n} \varphi(x)} \leqslant  p_{\psi, T}(f, \varphi, \gamma).
\]
Consequently, one has that  $P_{\psi, T}(f, \varphi, \varepsilon) \leqslant p_{\psi, T}(f, \varphi, \gamma)$.
\end{proof}

\begin{remark} \label{equivalent definition rmk}
(1) If \( \alpha, \gamma \) are open covers of \( X \) with \( \alpha < \gamma \), i.e., \( \gamma \) is a refinement of \( \alpha \), then
$
q_{\psi, T}(f, \varphi, \alpha) \leqslant q_{\psi, T}(f, \varphi, \gamma)
$; (2)Let $m=\inf \psi$ and $M=\sup \psi$. If \( d(x, y)<\operatorname{diam}(\alpha)\) implies that \( |\varphi(x)-\varphi(y)| \leqslant \delta \), then for each \( B \in \bigvee_{i=0}^{n-1} f^{-i} \alpha \),
\[
\sup _{x \in B} e^{S_{n} \varphi(x)} \leqslant e^{n\delta}\inf _{x \in B} e^{S_{n} \varphi(x)}.
\]
Consequently, one can show that
\[
\sum_{n \in S_{T}}\sum_{B \in \beta_{n}} \sup _{x \in B} e^{S_{n} \varphi(x)} \leqslant \sum_{n \in S_{T}}\sum_{B \in \beta_{n}} e^{n\delta} \inf _{x \in B} e^{S_{n} \varphi(x)}.
\]
Since \( n \leqslant T/m \), one can further show that
\[
p_{\psi, T}(f, \varphi, \alpha) \leqslant e^{T \delta / m} q_{\psi, T}(f, \varphi, \alpha).
\]
\end{remark}

The following theorem gives several equivalent definitions of induced topological pressure.

\begin{theorem} \label{open cover et}
Let \( (X, f) \) be a TDS, and \( \varphi, \psi \in C(X, \mathbb{R}) \) with \( 0<m \leqslant \psi(x) \leqslant M \) for each $x\in X$. Then each of the following equals \( P_{\psi}(\varphi) \):
\begin{enumerate}
\item[(i)] $\displaystyle \lim _{\delta \rightarrow 0}\sup _{\alpha}\limsup _{T \to \infty}\frac{1}{T} \log p_{\psi, T}(f, \varphi, \alpha)$, where the supremum is taken over all open cover $\alpha$ of $X$  with  $\operatorname{diam}(\alpha) \leqslant \delta$;
\item[(ii)] $\displaystyle 
    \lim _{k \rightarrow \infty}  \limsup _{T \rightarrow \infty}\frac{1}{T} \log p_{\psi, T}\left(f, \varphi,  \alpha_{k} \right) $, where $\left\{\alpha_{k}\right\}_{k \in \mathbb{N}}$ is a sequence of open covers of $X$  with $\operatorname{diam}\left(\alpha_{k}\right) \rightarrow 0$  as $k \rightarrow \infty $;
\item[(iii)] $\displaystyle \lim _{\delta \to 0}\sup_{\alpha} \limsup _{T\rightarrow \infty}\frac{1}{T} \log q_{\psi, T}(f, \varphi, \alpha)$, where the supremum  is taken over all open covers $\alpha$ of $X$ with  $\operatorname{diam}(\alpha) \leqslant \delta $;
\item[(iv)] $\displaystyle
\lim _{k \rightarrow \infty}   \limsup _{T \rightarrow \infty}\frac{1}{T} \log q_{\psi, T}\left(f, \varphi,  \alpha_{k} \right)$, where $\left\{\alpha_{k}\right\}_{k \in \mathbb{N}}$  is a sequence of open covers of $X$  with $\operatorname{diam}\left(\alpha_{k}\right) \rightarrow 0 \text{ as } k \rightarrow \infty $;
\item[(v)] $\displaystyle  \sup _{\alpha}\limsup _{T \rightarrow \infty}\frac{1}{T} \log q_{\psi, T}(f, \varphi, \alpha)$, where the supremum  is taken over all open covers of $X $.
\end{enumerate}
\end{theorem}

\begin{proof}
(i) For any \( \delta>0 \), let $\gamma$ be an arbitrary open cover of $X$ satisfying $\operatorname{diam}(\gamma) \leqslant \delta$. By Lemma \ref{equivalent definition lm} (ii), we have that
\[
P_{\psi, T}(f, \varphi, \delta) \leqslant p_{\psi, T}(f, \varphi, \gamma).
\]
Therefore, one has that
\begin{align*}
 \limsup _{T \rightarrow \infty} \frac{1}{T} \log P_{\psi, T}(f, \varphi, \delta) \leqslant \sup_{\gamma} \limsup _{T \rightarrow \infty}\frac{1}{T} \log p_{\psi, T}(f, \varphi, \gamma),
 \end{align*}
where the supremum is taken over all  open covers $\gamma $ of $X$  with $\operatorname{diam}(\gamma) \leqslant \delta$. Taking \( \delta \to 0 \), one has that \( P_{\psi}(\varphi) \) is not larger than the expression in (i).

Let \( \alpha \) be an  open cover of \( X \) with Lebesgue number \( \delta \). By Lemma \ref{equivalent definition lm} (i), we have that
\[
q_{\psi, T}(f, \varphi, \alpha) \leqslant P_{\psi, T}(f, \varphi, \delta / 2).
\] Let
$
\tau_{\alpha} := \sup \{|\varphi(x)-\varphi(y)| : d(x, y) \leqslant \operatorname{diam}(\alpha)\},
$
 by Remark \ref{equivalent definition rmk} (2) and the previous inequality, one has that
\[
p_{\psi, T}(f, \varphi, \alpha) \leqslant e^{T \tau_{\alpha} / m} q_{\psi, T}(f, \varphi, \alpha)\leqslant e^{T \tau_{\alpha} / m} P_{\psi, T}(f, \varphi, \delta / 2).
\]
Hence, we have that
\[
\limsup _{T \rightarrow \infty} \frac{1}{T} \log p_{\psi, T}(f, \varphi, \alpha) \leqslant \frac{\tau_{\alpha}}{m}+\limsup _{T \rightarrow \infty}\frac{1}{T} \log P_{\psi, T}(f, \varphi, \delta / 2).
\]
Since \( \varphi \) is uniformly continuous, \( \tau_{\alpha} \to 0 \) as \( \operatorname{diam}(\alpha) \to 0 \),  which yields that
\[
\lim _{\eta \rightarrow 0}\sup _{\alpha}\Big\{\limsup _{T \rightarrow \infty} \frac{1}{T} \log p_{\psi, T}(f, \varphi, \alpha): \operatorname{diam}(\alpha) \leqslant \eta \Big\} \leqslant P_{\psi}(\varphi).
\]
Thus, (i) is proved. The second statement can be proven in a similar fashion.

(iii) By the definitions and Remark \ref{equivalent definition rmk}(2), one has that
\[
e^{-T \tau_{\alpha} / m}p_{\psi, T}(f, \varphi, \alpha) \leqslant q_{\psi, T}(f, \varphi, \alpha)\leqslant p_{\psi, T}(f, \varphi, \alpha).
\]
Hence,  we have that
\[
\begin{aligned}
\frac{-\tau_{\alpha}}{m}+\limsup _{T \rightarrow \infty} \frac{1}{T} \log p_{\psi, T}(f, \varphi, \alpha) & \leqslant \limsup _{T \rightarrow \infty} \frac{1}{T} \log q_{\psi, T}(f, \varphi, \alpha) \\
& \leqslant \limsup _{T \rightarrow \infty} \frac{1}{T} \log p_{\psi, T}(f, \varphi, \alpha).
\end{aligned}
\]
Since \( \varphi \) is uniformly continuous, \( \tau_{\alpha} \to 0 \) as \( \operatorname{diam}(\alpha) \to 0 \). This together with the first statement implies the desired result.

The statement in  (iv) can be proven in a similar fashion.

(v) Let $\alpha$ be an open cover of $X$, and let $2\varepsilon$ denote the Lebesgue number of $\alpha$. By Lemma \ref{equivalent definition lm} (i), one has that
\[
q_{\psi, T}(f, \varphi, \alpha) \leqslant Q_{\psi, T}(f, \varphi, \varepsilon).
\]
Therefore,
\[
\limsup _{T \rightarrow \infty}\frac{1}{T}\log q_{\psi, T}(f, \varphi, \alpha) \leqslant  \limsup _{T \rightarrow \infty}\frac{1}{T}\log Q_{\psi, T}(f, \varphi, \varepsilon) \le P_{\psi}(\varphi).
\]
Since the open cover $\alpha$ is chosen arbitrarily, one has that
\[
\sup_{\alpha}\limsup _{T \rightarrow \infty}\frac{1}{T} \log q_{\psi, T}(f, \varphi, \alpha) \le P_{\psi}(\varphi).
\]
For each $\delta>0$,  we have that
\[
\sup_{\alpha}\limsup _{T \rightarrow \infty}\frac{1}{T} \log q_{\psi, T}(f, \varphi, \alpha) \geqslant \sup_{\gamma}\limsup _{T \rightarrow \infty}\frac{1}{T} \log q_{\psi, T}(f, \varphi, \gamma),
\]
where the supremum on the RHS of the above inequality is taken over all open covers $\gamma$ of $X$ with $\operatorname{diam}(\gamma) \le \delta$. Letting $\delta \to 0$, it follows from (iii) that
\[
\sup_{\alpha}\limsup _{T \rightarrow \infty}\frac{1}{T} \log q_{\psi, T}(f, \varphi, \alpha) \geqslant P_{\psi}(\varphi).
\]
Thus, (v) is proved.
\end{proof}
	
	\subsection{Properties of induced topological pressure} \label{induce prop}
We first recall the variational principle for the induced topological pressure in \cite{xing2015induced}.

\begin{theorem}\label{induced Variational Principle}
Let \( (X, f) \) be a TDS, and \( \varphi, \psi \in C(X, \mathbb{R}) \) with \( \psi > 0 \). Then
\[
P_{\psi}(\varphi) = \sup \Big\{ \frac{h_{\nu}(f)}{\int \psi \, d \nu} + \frac{\int \varphi \, d \nu}{\int \psi \, d \nu} : \nu \in \mathcal{M}( X,f ) \Big\},
\]
where $h_{\nu}(f)$ denotes the measure theoretic entropy of $f$ w.r.t. $\nu$, see \cite{walters1982introduction} for more detailed descriptions.
\end{theorem}

\begin{remark}\label{finite top}
By the variational principle, one can easily show that
$$ h_{top}(f) < \infty  \iff P_{\psi}(\varphi) < \infty $$
for every $\varphi\in C(X, \mathbb{R})$.
\end{remark}

Next we study the properties of induced topological pressure functional  \( P_{\psi}(\cdot): C(X, \mathbb{R}) \rightarrow \mathbb{R} \cup\{\infty\} \). Two continuous functions $\varphi_{1}, \varphi_{2} \in C(X, \mathbb{R})$ are called  cohomologous, and is denoted by $\varphi_{1} \sim \varphi_{2}$, if there exists  $h \in C(X, \mathbb{R})$ such that $\varphi_{1}=\varphi_{2}+h-h \circ f$.

\begin{theorem} \label{induced pressure prop}
Let \( (X, f) \) be a TDS, \( \varphi, \psi,g \in C(X, \mathbb{R}) \) with \( \psi > 0 \), and \( c \in \mathbb{R} \). Then the following properties hold:
\begin{enumerate}
\item[(i)] (Monotonicity) If \( \varphi_{1} \leqslant \varphi_{2} \), then \( P_{\psi}(\varphi_{1}) \leqslant P_{\psi}(\varphi_{2}) \).
\item[(ii)] (Continuity) The function
$
P_{\psi}(\cdot): C(X, \mathbb{R})  \rightarrow \mathbb{R} \cup\{\infty\}
$
is continuous.
\item[(iii)] (Bounds estimate) Let \( m = \inf \psi \) and \( M = \sup \psi \), we have
\[
P_{\psi}(\varphi)+\frac{\inf g}{M} \leqslant P_{\psi}(\varphi+g) \leqslant P_{\psi}(\varphi)+\frac{\sup g}{m}.
\]
\item[(iv)] (Convexity) If \( P_{\psi}(\cdot)<\infty \), then \( P_{\psi}(\cdot) \) is convex.
\item[(v)] If \( c \geqslant 1 \), then \( P_{\psi}( c \varphi) \leqslant c P_{\psi}( \varphi) \); if \( c \leqslant 1 \), then \( P_{\psi}( c \varphi) \geqslant c P_{\psi}( \varphi) \).
\item[(vi)] \( P_{\psi}( \varphi+g) \leqslant P_{\psi}( \varphi)+P_{\psi}( g) \).
\item[(vii)]    If $\varphi_{1} \sim \varphi_{2}$, then $P_{\psi}\left(\varphi_{1}\right)=P_{\psi}\left(\varphi_{2}\right)$.
In particular, if \( \varphi \sim \psi \), then for each \( t \in \mathbb{R} \), we have  \( P_{\psi}(t \varphi)=t+P_{\psi}(0) \).
\end{enumerate}
\end{theorem}

\begin{remark} By $\mathrm{(i)}$ and $\mathrm{(v)}$, one has that  \( |P_{\psi}( \varphi)| \leqslant P_{\psi}(|\varphi|) \).
\end{remark}

	\begin{proof}

The statement of (i) follows immediately from the definition.

(ii) By \cite[Corollary 3.2]{xing2015induced}, we have
\[
P_{\psi}(\varphi)=\inf \left\{\beta: P_{\mathrm{top}}(f,\varphi-\beta \psi) \leqslant 0\right\}=\sup \left\{\beta: P_{\mathrm{top}}(f,\varphi-\beta \psi) \geqslant 0\right\}.
\]
Since the mapping \( P_{\mathrm{top}}(f,\cdot): C(X, \mathbb{R}) \rightarrow \mathbb{R} \cup\{\infty\} \) is continuous, the conclusion holds.

(iii) By Theorem \ref{induced Variational Principle} and (i), we have that
\begin{align*}
P_{\psi}(\varphi+g) & = \sup \Big\{ \frac{h_{\nu}(f)}{\int \psi \, d \nu} + \frac{\int \varphi+g \, d \nu}{\int \psi \, d \nu} : \nu \in \mathcal{M}\left( X,f \right) \Big\} \\
&\leqslant \sup \Big\{ \frac{h_{\nu}(f)}{\int \psi \, d \nu} + \frac{\int \varphi \, d \nu}{\int \psi \, d \nu} : \nu \in \mathcal{M}\left( X,f \right) \Big\}+\frac{\sup g}{m} \\
&= P_{\psi}(\varphi)+\frac{\sup g}{m}.
\end{align*}
Similarly, one can show that
\begin{align*}
P_{\psi}(\varphi+g) \ge P_{\psi}(\varphi)+\frac{\inf g}{M}.
\end{align*}
This completes the proof of (iii).

(iv) Given \( \varphi_{1}, \varphi_{2} \in C\left(X, \mathbb{R}\right) \) with \( P_{\psi}(\varphi_{1})<\infty \) and \( P_{\psi}(\varphi_{2})<\infty \), and  \( t \in[0,1] \), then it follows from Theorem \ref{induced Variational Principle} that
\begin{align*}
P_{\psi}(t \varphi_{1}+(1-t) \varphi_{2}) & = \sup \Big\{ \frac{h_{\nu}(f)}{\int \psi \, d \nu} + \frac{\int t \varphi_{1}+(1-t) \varphi_{2} \, d \nu}{\int \psi \, d \nu} : \nu \in \mathcal{M}\left( X,f \right) \Big\} \\
& \leqslant t \sup \Big\{ \frac{h_{\nu}(f)}{\int \psi \, d \nu} + \frac{\int \varphi_{1} \, d \nu}{\int \psi \, d \nu} : \nu \in \mathcal{M}\left( X,f \right) \Big\} + \\
& \quad (1-t) \sup \Big\{ \frac{h_{\nu}(f)}{\int \psi \, d \nu} + \frac{\int \varphi_{2} \, d \nu}{\int \psi \, d \nu} : \nu \in \mathcal{M}\left( X,f \right) \Big\} \\
& = t P_{\psi}(\varphi_{1}) + (1-t) P_{\psi}(\varphi_{2}).
\end{align*}
Thus, the convexity is established.

(v) For \( c \geqslant 1 \), by Theorem \ref{induced Variational Principle}, we have that
\begin{align*}
P_{\psi}( c\varphi ) & = \sup \left\{ \frac{h_{\nu}(f)}{\int \psi \, d \nu} + \frac{c\int \varphi \, d \nu}{\int \psi \, d \nu} : \nu \in \mathcal{M}( X,f ) \right\} \\
& \leqslant c \sup \left\{ \frac{h_{\nu}(f)}{\int \psi \, d \nu} + \frac{\int \varphi \, d \nu}{\int \psi \, d \nu} : \nu \in \mathcal{M}\left( X,f \right) \right\} \\
& = c P_{\psi}( \varphi).
\end{align*}
 Similarly, for \( 0\leqslant c \leqslant 1 \), one can show that
\begin{align*}
P_{\psi}( c\varphi ) \ge c P_{\psi}( \varphi).
\end{align*}
For \( c\leqslant 0  \), we have that
\begin{align*}
P_{\psi}( c\varphi ) & = \sup \left\{ \frac{h_{\nu}(f)}{\int \psi \, d \nu} + \frac{c\int \varphi \, d \nu}{\int \psi \, d \nu} : \nu \in \mathcal{M}\left( X,f \right) \right\} \\
& \geqslant  \inf \left\{ \frac{ch_{\nu}(f)}{\int \psi \, d \nu} + \frac{c\int \varphi \, d \nu}{\int \psi \, d \nu} : \nu \in \mathcal{M}\left( X,f \right) \right\} \\
& = c \sup \left\{ \frac{h_{\nu}(f)}{\int \psi \, d \nu} + \frac{\int \varphi \, d \nu}{\int \psi \, d \nu} : \nu \in \mathcal{M}\left( X,f \right) \right\} \\
& = c P_{\psi}( \varphi).
\end{align*}
This completes the proof of (v).

(vi) Given $T > 0$, $\varepsilon > 0$ and $n \in S_{T}$, if $F_{n}$ is an $(n,\varepsilon)$-spanning set for $X_{n}$, then  we have that
\[
\sum_{n \in S_{T}} \sum_{x \in F_{n}} \exp (S_{n} (\varphi+ g)(x)) \leqslant \Big( \sum_{n \in S_{T}} \sum_{x \in F_{n}} e^{S_{n} \varphi(x)}\Big) \Big(\sum_{n \in S_{T}} \sum_{x \in F_{n}} e^{S_{n} g(x)} \Big).
\]
Hence, we have
\[
P_{T,\psi}(f, \varphi+g, \varepsilon) \leqslant P_{T,\psi}(f,\varphi, \varepsilon) \cdot P_{T,\psi}(f,g, \varepsilon),
\]
which yields that \( P_{\psi}( \varphi+g) \leqslant P_{\psi}( \varphi)+P_{\psi}( g) \).

(vii)
If $\varphi_{1} \sim \varphi_{2}$, i.e., there exists $h \in C(X, \mathbb{R})$ such that $\varphi_{1}=\varphi_{2}+h-h \circ f$,  by Theorem \ref{induced Variational Principle} one has that $P_{\psi}\left(\varphi_{1}\right)=P_{\psi}\left(\varphi_{2}\right)$.

If \( \varphi \sim \psi \), then \( t\varphi \sim t\psi \) for every \( t \in \mathbb{R}\). Hence, \( P_{\psi}(t \varphi) = P_{\psi}(t \psi) \). Furthermore, we have that
\begin{align*}
P_{\psi}( t\psi ) & = \sup \left\{ \frac{h_{\nu}(f)}{\int \psi \, d \nu} + \frac{t\int \psi \, d \nu}{\int \psi \, d \nu} : \nu \in \mathcal{M}( X,f ) \right\} \\
& = \sup \left\{ \frac{h_{\nu}(f)}{\int \psi \, d \nu} + t : \nu \in \mathcal{M}( X,f ) \right\}\\
 &= t + P_{\psi}( 0 ).
\end{align*}
This completes the proof of (vii).
\end{proof}

The following result shows that the induced topological pressure can determine the members of \( \mathcal{M}(X, f) \).


\begin{proposition} \label{prop2.1}
Let \( (X, f) \) be a TDS with \( h_{\mathrm{top}}(f) < +\infty \), and let \(  \psi \in C(X, \mathbb{R}) \) with \( \psi > 0 \). Let \( \mu: \mathcal{B}(X) \to \mathbb{R} \) be a probability measure on $X$ (where $\mathcal{B}(X)$ is the Borel $\sigma$-algebra of $X$). Then for every \( \varphi \in C(X, \mathbb{R}) \), we have that
$
\int \varphi  d \mu \le P_{\psi}(\varphi) \int \psi  d \mu
$
if and only if \( \mu \in \mathcal{M}(X, f) \).
\end{proposition}	

\begin{proof}
If \( \mu \in \mathcal{M}(X, f) \), then by Theorem \ref{induced Variational Principle}, for every \( \varphi \in C(X, \mathbb{R}) \), one has that \( \int \varphi \, d \mu \leqslant P_{\psi}(\varphi) \int \psi \, d \mu \).

Now assume that \( \mu \) is a probability measure on $X$, and for every \( \varphi \in C(X, \mathbb{R}) \), we have that \( \int \varphi \, d \mu \leqslant P_{\psi}(\varphi) \int \psi \, d \mu \).

If \( n \in \mathbb{Z} \) and \( \varphi \in C(X, \mathbb{R}) \), by  Remark \ref{finite top} and Theorem \ref{induced Variational Principle}, we have that
\[
n \int (\varphi \circ f - \varphi) \, d \mu \leqslant P_{\psi}( n\varphi \circ f -n \varphi) \int \psi \, d \mu = P_{\psi}(0) \int \psi \, d \mu  < +\infty.
\]
If \( n \geqslant 1 \), one has that
\[
\int (\varphi \circ f - \varphi) \, d \mu \leqslant P_{\psi}(0) \frac{\int \psi \, d \mu}{n}.
\]
This yields that  \( \int (\varphi \circ f - \varphi) \, d \mu \leqslant 0 \) by letting \( n \to +\infty \).

If \( n <0 \),  one has that
\[
\int (\varphi \circ f - \varphi) \, d \mu \geqslant P_{\psi}(0) \frac{\int \psi \, d \mu}{n},
\]
which implies that \( \int (\varphi \circ f - \varphi) \, d \mu \geqslant 0 \).

Hence,  \( \int \varphi \circ f \, d \mu = \int \varphi \, d \mu \) for every \( \varphi \in C(X, \mathbb{R}) \).  So \( \mu \in \mathcal{M}(X, f) \).
\end{proof}

The following result gives a sufficient condition for the upper semi-continuity of the entropy map by using the induced topological pressure.

\begin{proposition} \label{prop2.2}
Let \( (X, f) \) be a TDS with \( h_{\mathrm{top}}(f) < +\infty \), \( \psi \in C(X, \mathbb{R}) \) with \( \psi > 0 \),  and let \( \mu_{0} \in \mathcal{M}(X, f) \). If
\[
h_{\mu_{0}}(f) = \inf \left\{ P_{\psi}(\varphi)\int \psi \, d \mu_{0}  - \int \varphi \, d \mu_{0} : \varphi \in C(X, \mathbb{R}) \right\},
\]
then the entropy map is upper semi-continuous at \( \mu_{0} \).
\end{proposition}
\begin{proof}
For each \( \varepsilon > 0 \), there exists \( g \in C(X, \mathbb{R}) \) such that
\[
P_{\psi}(g)\int \psi \, d \mu_{0}  - \int g \, d \mu_{0} < h_{\mu_{0}}(f) + \frac{\varepsilon}{2}.
\]
Let
\[
V_{\mu_{0}}(g, \varepsilon)=\left\{\mu \in \mathcal{M}(X, f) : \left| \int g \, d \mu - \int g \, d \mu_{0} \right| < \varepsilon,\, | P_{\psi}(g)|\left|\int \psi \, d \mu - \int \psi \, d \mu_{0} \right| < \varepsilon\right\}.
\]
For each \( \mu \in V_{\mu_{0}}(g, \varepsilon) \), by Theorem \ref{induced Variational Principle}, we have that
\[
\begin{aligned}
h_{\mu_{0}}(f) & \geqslant P_{\psi}(g) \int \psi \, d \mu_{0} - \int g \, d \mu_{0} - \varepsilon \\
& > P_{\psi}(g) \int \psi \, d \mu - \int g \, d \mu - 3\varepsilon \\
& > h_{\mu}(f) - 3\varepsilon.
\end{aligned}
\]
Thus, the entropy map is upper semi-continuous at \( \mu_{0} \).
\end{proof}

\section{Equilibrium state, freezing state of induced topological pressure}\label{ES and FS}
This section  first gives some basic properties of equilibrium state and subdifferential of induced topological pressure.  Moreover, as in \cite{hedges2024equivalence},   we introduce the concept of freezing states and zero-temperature limits of induced topological pressure, and investigate the relationship between equilibrium states and freezing states.

\subsection{Equilibrium states}
Let \( (X, f) \) be a TDS, and let \( \varphi, \psi \in C(X, \mathbb{R}) \) with \( \psi > 0 \). An invariant measure \( \mu\in \mathcal{M}(X, f) \) is called an \emph{equilibrium state} of \( P_{\psi}(\varphi) \), if  \(\displaystyle P_{\psi}(\varphi) = \frac{h_{\mu}(f) + \int \varphi \, d\mu}{\int \psi \, d\mu} \). Let \( \mathcal{M}_{\psi, \varphi}(X, f) \) denote the set of all equilibrium states of \( P_{\psi}(\varphi) \).

The following result gives some description of equilibrium states.

\begin{theorem}
Let \( (X, f) \) be a TDS with \( h_{top}(f) < +\infty \), and let \( \varphi, \psi \in C(X, \mathbb{R}) \) with \( \psi > 0 \). The following properties hold:
\begin{enumerate}
\item[(i)] \( \mathcal{M}_{\psi, \varphi}(X, f) \) is convex;
\item[(ii)] The extreme points of \( \mathcal{M}_{\psi, \varphi}(X, f) \) are exactly the ergodic measures in \( \mathcal{M}_{\psi, \varphi}(X, f) \);
\item[(iii)] If \( \mathcal{M}_{\psi, \varphi}(X, f) \neq \emptyset \), then \( \mathcal{M}_{\psi, \varphi}(X, f) \) contains an ergodic measure.
\end{enumerate}


\end{theorem}

\begin{proof}
(i) For every \( \mu, \nu \in \mathcal{M}_{\psi, \varphi}(X, f) \) and every \( 0\le \lambda \le 1 \),
  we have that
  \[
  \lambda \int \psi  d\mu \cdot P_{\psi}(\varphi)=\lambda \Big( h_{\mu}(f) + \int \varphi \, d\mu \Big)
  \]
  and
  \[(1 - \lambda)\int \psi  d\nu \cdot P_{\psi}(\varphi)= (1 - \lambda)\Big( h_{\nu}(f) + \int \varphi \, d\nu \Big).
  \]
  Take the sum of the previous two equalities, one has that
\begin{align*}
\frac{h_{\lambda \mu + (1 - \lambda) \nu}(f) + \int \varphi \, d(\lambda \mu + (1 - \lambda) \nu)}{\int \psi \, d(\lambda \mu + (1 - \lambda) \nu)} = P_{\psi}(\varphi).
\end{align*}
Thus, \( \lambda \mu + (1 - \lambda) \nu \in \mathcal{M}_{\psi, \varphi}(X, f) \). Hence, \( \mathcal{M}_{\psi, \varphi}(X, f) \) is convex.

(ii) Let \( \mu \) be an ergodic measure in \( \mathcal{M}_{\psi, \varphi}(X, f) \). Then, \( \mu \) is an extreme point of \( \mathcal{M}(X, f) \). Since $\mathcal{M}_{\psi,\varphi}(X, f) \subset \mathcal{M}(X, f)$, the measure \( \mu \) is an extreme point of \( \mathcal{M}_{\psi, \varphi}(X, f) \).

Now assume that \( \mu \) is an extreme point in \( \mathcal{M}_{\psi, \varphi}(X, f) \subset \mathcal{M}(X, f) \), and \( \mu = p \mu_1 + (1 - p) \mu_2 \) for some \( 0\le p \le 1\), \( \mu_1, \mu_2 \in \mathcal{M}(X, f) \). Hence, we have
\[
P_{\psi}(\varphi)=\frac{h_{\mu}(f) + \int \varphi \, d\mu}{\int \psi \, d\mu} = \frac{p \left( h_{\mu_1}(f) + \int \varphi \, d\mu_1 \right) + (1 - p) \left( h_{\mu_2}(f) + \int \varphi \, d\mu_2 \right)}{p \int \psi \, d\mu_1 + (1 - p) \int \psi \, d\mu_2}
\]
and
\[
P_{\psi}(\varphi) \geqslant \frac{h_{\mu_i}(f) + \int \varphi \, d\mu_i}{\int \psi \, d\mu_i} \quad i = 1, 2,
\]
which implies that
\[
P_{\psi}(\varphi) = \frac{h_{\mu_i}(f) + \int \varphi \, d\mu_i}{\int \psi \, d\mu_i} \quad  i = 1, 2.
\]
Thus, \( \mu_1, \mu_2 \in \mathcal{M}_{\psi, \varphi}(X, f) \).  This implies that $\mu= \mu_{1}=\mu_{2}$, since $\mu$ is an extreme point of $\mathcal{M}_{\psi,\varphi}(X, f)$. Hence, $\mu$ is an ergodic measure.

(iii) Given \( \mu \in \mathcal{M}_{\psi, \varphi}(X, f) \), consider the ergodic decomposition of \( \mu = \int_{\mathcal{E}(X, f)} m \, d\tau(m) \), we have that
\[
\int_{\mathcal{E}(X, f)} \Big( h_m(f) + \int \varphi \, dm \Big) d\tau(m)= P_{\psi}( \varphi) \int_{\mathcal{E}(X, f)}  \int \psi \, dm  d\tau(m).
\]
By the variational principle of induced topological pressure (see Theorem \ref{induced Variational Principle}), for every \( m \in \mathcal{E}(X, f) \), we have that
\[
h_m(f) + \int \varphi \, dm \leqslant P_{\psi}( \varphi) \int \psi \, dm.
\]
This yields that
\[
h_m(f) + \int \varphi \, dm = P_{\psi}( \varphi) \int \psi \, dm, \quad \tau-\text{a.e. }m \in\mathcal{E}(X, f).
\]
This completes the proof of the third statement.

\end{proof}

Given \( \varphi, \psi, g \in C(X, \mathbb{R}) \) with \( \psi > 0 \). Since the map $t\mapsto P_{\psi}(\varphi + t g)$ is convex by Theorem \ref{induced pressure prop}, one can show that the map $t \mapsto t^{-1} \left( P_{\psi}(\varphi + t g) - P_{\psi}(\varphi) \right)$ is increasing. Thus, the following two limits exist:
\[
d^{+} P_{\psi}( \varphi)(g): = \lim_{t \rightarrow 0+} \frac{1}{t} \left( P_{\psi}( \varphi + t g) - P_{\psi}( \varphi) \right)
\]
and
\[
d^{-} P_{\psi}( \varphi)(g) := \lim_{t \rightarrow 0-} \frac{1}{t} \left( P_{\psi}( \varphi + t g) - P_{\psi}( \varphi) \right).
\]
Furthermore, one can show that the following properties hold:
\begin{enumerate}
\item[(1)] \( d^{-} P_{\psi}( \varphi)(g) = -d^{+} P_{\psi}( \varphi)(-g) \);
\item[(2)] \( d^{-} P_{\psi}(\varphi)(g) \leqslant d^{+} P_{\psi}(\varphi)(g) \);
\item[(3)] If \( \lambda \geqslant 0 \), then \( d^{+} P_{\psi}(\varphi)(\lambda g) = \lambda d^{+} P_{\psi}(\varphi)(g) \).
\end{enumerate}
The induced topological pressure of \( P_{\psi}( \varphi) \)  is said to be G\^{a}teaux differentiable, if for all \( g \in C(X, \mathbb{R}) \), the following limit exists:
\[
\lim_{t \rightarrow 0} \frac{P_{\psi}( \varphi + t g) - P_{\psi}( \varphi)}{t}.
\]
The previous limit exists if and only if \( d^{+} P_{\psi}( \varphi)(g) = -d^{+} P_{\psi}( \varphi)(-g) \) holds for all \( g \in C(X,\mathbb{R}) \). 

\begin{definition}
 A probability measure \( \mu: \mathcal{B}(X) \rightarrow \mathbb{R} \)  is called the tangent functional (also known as the subdifferential) of \( P_{\psi}(\cdot) \) at \( \varphi \), if
 \[
P_{\psi}( \varphi + g) - P_{\psi}( \varphi) \geqslant \frac{\int g \, d \mu}{\int \psi \, d \mu}
\]
for all $g \in C(X, \mathbb{R})$.
Let \( t_{\psi, \varphi}(X, f) \) denote the set of all tangent functionals of \( P_{\psi}(\cdot) \) at \( \varphi \).
\end{definition}

The following theorem gives some basic properties of the tangent functional.

\begin{proposition}
Let \( (X, f) \) be a TDS with \( h_{top}(f) < +\infty \), and let \( \varphi, \psi \in C(X, \mathbb{R}) \) with \( \psi > 0 \). The following properties hold:
\begin{enumerate}
\item[(i)] \( t_{\psi, \varphi}(X, f) \) is a compact convex set;
\item[(ii)] \( \mathcal{M}_{\psi, \varphi}(X, f) \subset t_{\psi, \varphi}(X, f) \subset \mathcal{M}(X, f) \);
\item[(iii)] If the entropy map \(\mu\mapsto  h_\mu(f) \) is upper semi-continuous on $\mathcal{M}(X, f)$, then \( \mathcal{M}_{\psi, \varphi}(X, f) \) is non-empty, and thus \( t_{\psi, \varphi}(X, f) \neq \emptyset \);
\item[(iv)] \( \mu \in t_{\psi, \varphi}(X, f) \) if and only if
\[
P_{\psi}(\varphi) - \frac{\int \varphi \, d\mu}{\int \psi \, d\mu} = \inf \Big\{ P_{\psi}(h) - \frac{\int h \, d\mu}{\int \psi \, d\mu} : h \in C(X) \Big\}
\]
holds;
\item[(v)] \( d^{+} P_{\psi}(\varphi)(g) \ge \sup \left\{ \frac{\int g \, d\mu}{\int \psi \, d\mu} : \mu \in t_{\psi, \varphi}(X, f) \right\} \) for every $g\in C(X,\mathbb{R})$.
\end{enumerate}
\end{proposition}
\begin{proof}
(i)  By definition, the compactness clearly holds. For every $\mu_1, \mu_2 \in t_{\psi,\varphi}(X, f)$ and every $g \in C(X, \mathbb{R})$, we have that
\[
P_{\psi}( \varphi + g) - P_{\psi}( \varphi) \geqslant \frac{\int g \, d \mu_i}{\int \psi \, d \mu_i} \quad i=1,2.
\]
For each $p \in [0,1]$, consider the measure $p\mu_1 + (1-p)\mu_2$. One can show that
\[
P_{\psi}( \varphi + g) - P_{\psi}( \varphi) \geqslant \frac{p\int g \, d \mu_{1}+(1-p)\int g \, d \mu_{2}}{p\int \psi \, d \mu_{1}+(1-p)\int \psi \, d \mu_{2}} = \frac{\int g \, d (p\mu_{1}+(1-p)\mu_{2})}{\int \psi \, d (p\mu_{1}+(1-p)\mu_{2})}.
\]
Therefore, $p\mu_{1}+(1-p)\mu_{2} \in t_{\psi,\varphi}(X, f)$.

(ii) Let \( \mu \in \mathcal{M}_{\psi, \varphi}(X, f) \),  \( g \in C(X, \mathbb{R}) \). By Theorem \ref{induced Variational Principle}, we have that
\[
P_{\psi}(\varphi + g) - P_{\psi}(\varphi) \geqslant \frac{h_{\mu}(f) + \int \varphi \, d\mu + \int g \, d\mu}{\int \psi \, d\mu} - \frac{h_{\mu}(f) + \int \varphi \, d\mu}{\int \psi \, d\mu} = \frac{\int g \, d\mu}{\int \psi \, d\mu}.
\]
Thus, \( \mathcal{M}_{\psi, \varphi}(X, f) \subset t_{\psi, \varphi}(X, f) \).

Next, we prove that \( t_{\psi, \varphi}(X, f) \subset \mathcal{M}(X, f) \). Let \( \mu \in t_{\psi, \varphi}(X, f) \). For every \( n \in \mathbb{Z} \) and every \( g \in C(X, \mathbb{R}) \),  by Theorem \ref{induced pressure prop} (vii), we have that
\[
\begin{aligned}
\frac{n \int (g \circ f - g) \, d\mu}{\int \psi d\mu} &\leqslant P_{\psi}(\varphi + n(g \circ f - g)) - P_{\psi}(\varphi) = 0.
\end{aligned}
\]
The arbitrariness of $n\in \mathbb{Z}$ implies that \( \int g \circ f \, d\mu = \int g \, d\mu \). Hence, we have that \( \mu \in \mathcal{M}(X, f) \).

(iii) By (ii) and Theorem \ref{induced Variational Principle}, this is obviously true.

(iv) It follows from the definition of tangent functional that
\begin{align*}
& \mu \in t_{\psi,\varphi}(X, f) \\
& \iff P_{\psi}(\varphi)-\frac{\int \varphi d \mu}{\int \psi d \mu} \leqslant P_{\psi}(\varphi+g)-\frac{\int \varphi+g d \mu}{\int \psi d \mu} \quad \forall g \in C(X, \mathbb{R})\\
& \iff P_{\psi}(\varphi) - \frac{\int \varphi \, d\mu}{\int \psi \, d\mu} = \inf \Big\{ P_{\psi}(h) - \frac{\int h \, d\mu}{\int \psi \, d\mu} \mid h \in C(X,\mathbb{R}) \Big\}.
\end{align*}

(v) Take \( \mu \in t_{\psi, \varphi}(X, f) \) and $g\in C(X,\mathbb{R})$, for every \( t > 0 \), we have that
\[
\frac{\int g \, d\mu}{\int \psi \, d\mu} \leqslant t^{-1} \left( P_{\psi}(\varphi + t g) - P_{\psi}(\varphi) \right).
\]
Letting $t \to 0^{+}$, one has that
\[
\frac{\int g \, d\mu}{\int \psi \, d\mu} \leqslant d^{+} P_{\psi}(\varphi)(g).
\]
Since \( \mu \in t_{\psi, \varphi}(X, f) \) is chosen arbitrarily, the desired result follows immediately.
\end{proof}

\begin{theorem}
Let \( (X, f) \) be a TDS with \( h_{top}(f) < +\infty \), and let \( \varphi_1, \varphi_2, \psi \in C(X, \mathbb{R}) \) with \( \psi > 0 \). If $t_{\psi,\varphi_1}(X, f) \cap t_{\psi,\varphi_2}(X, f)\neq \emptyset$, then
\[
P_{\psi}\left( p \varphi_1 + (1-p) \varphi_2 \right) = P_{\psi}(\varphi_1) + (1-p) \frac{\int (\varphi_2 - \varphi_1) \, dm}{\int \psi \, dm},\quad \forall\, 0\le p \le 1
\]
for every $m \in t_{\psi,\varphi_1}(X, f) \cap t_{\psi,\varphi_2}(X, f)$,
and
$
t_{\psi,p \varphi_1 + (1 - p) \varphi_2}(X, f) \subset t_{\psi,\varphi_1}(X, f) \cap t_{\psi,\varphi_2}(X, f)
$ for every $0<p<1$.
\end{theorem}
\begin{proof}
Let \( m \in t_{\psi,\varphi_1}(X, f) \cap t_{\psi,\varphi_2}(X, f) \)  and \( 0 \le p \le 1 \). By Theorem \ref{induced pressure prop} (iv), we have that
\begin{equation}\label{3*}
\begin{aligned}
p P_{\psi}(\varphi_1) + (1 - p) P_{\psi}(\varphi_2) &\geqslant P_{\psi}(p \varphi_1 + (1 - p) \varphi_2) \\
&= P_{\psi}\left( \varphi_1 + (1 - p)(\varphi_2 - \varphi_1) \right)\\
 &\geqslant P_{\psi}(\varphi_1) + (1 - p) \frac{\int (\varphi_2 - \varphi_1) \, dm}{\int \psi \, dm}.
\end{aligned}
\end{equation}
Consequently, we have that
\[
P_{\psi}(\varphi_2) - \frac{\int \varphi_2 \, dm}{\int \psi \, dm} \geqslant P_{\psi}(\varphi_1) - \frac{\int \varphi_1 \, dm}{\int \psi \, dm}.
\]
The reverse inequality can be shown in a similar fashion since \( m \in t_{\psi,\varphi_1}(X, f) \cap t_{\psi,\varphi_2}(X, f) \). Hence, we have that
\begin{align}\label{3***}
P_{\psi}(\varphi_2) - \frac{\int \varphi_2 \, dm}{\int \psi \, dm} = P_{\psi}(\varphi_1) - \frac{\int \varphi_1 \, dm}{\int \psi \, dm}.
\end{align}
This together with inequality \eqref{3*} yields that
\begin{align} \label{3**}
P_{\psi}(p \varphi_1 + (1 - p) \varphi_2) = P_{\psi}(\varphi_1) + (1 - p) \frac{\int \varphi_2 - \varphi_1 \, dm}{\int \psi \, dm}.
\end{align}

Given $0<p<1$, let $\varphi_p=p \varphi_1 + (1 - p) \varphi_2$.  
Since $t_{\psi,\varphi_1}(X, f) \cap t_{\psi,\varphi_2}(X, f)\neq \emptyset$, by \eqref{3***} and \eqref{3**} we have that
\[
P_{\psi}(\varphi_p) = pP_{\psi}(\varphi_1)+(1-p)P_{\psi}(\varphi_2).
\]
Let \( \mu \in t_{\psi,\varphi_p}(X, f) \) and $g\in C(X, \mathbb{R})$, by Theorem \ref{induced pressure prop} (iv) we have that
\begin{align*}
\int g d\mu &\le P_{\psi}(\varphi_p+g)-P_{\psi}(\varphi_p)\\
&\le pP_{\psi}(\varphi_1+\frac{g}{p})+(1-p)P_{\psi}(\varphi_2)-( pP_{\psi}(\varphi_1)+(1-p)P_{\psi}(\varphi_2))\\
&=p(P_{\psi}(\varphi_1+\frac{g}{p})-P_{\psi}(\varphi_1)).
\end{align*}
Hence, we have that $\int \frac{g}{p} d\mu\le P_{\psi}(\varphi_1+\frac{g}{p})-P_{\psi}(\varphi_1)$, which implies that $\mu\in t_{\psi,\varphi_1}(X, f)$. Similarly, one can show that  $\mu\in t_{\psi,\varphi_2}(X, f)$. This completes the proof of the theorem.
\end{proof}

\subsection{Freezing States}  We extend the results established in \cite{hedges2024equivalence} for induced topological pressure . Precisely,  we introduce the concept of freezing states and zero-temperature limits of induced topological pressure, and investigate the relationship between equilibrium states and freezing states.

Let \( (X, f) \) be a TDS, and let \( \varphi, \psi \in C(X, \mathbb{R}) \) with \( \psi > 0 \).  For every \( \beta > 0 \), assume that $P_\psi(\beta \varphi)$ has an equilibrium state  \( \mu_{\beta} \),  \( \mu_{\infty} \) is called the \emph{zero temperature limit} of \( \varphi \), if  \( \left\{\mu_{\beta}\right\} \) converges to \( \mu_{\infty} \) in the weak-* sense as \( \beta \to \infty \). The function \( \varphi \) is said to \emph{freeze} at \( \beta_0 \), if there exists an invariant measure \( \mu \) such that  \( \mu \) is the equilibrium state of the induced topological pressure $P_\psi(\beta \varphi)$  for all \( \beta \geqslant \beta_0 \), and the measure  \( \mu \) is called a \emph{freezing state} of  \( \varphi \). It is clear that \( \mu \) is the zero temperature limit of \( \varphi \), if \( \mu \) is a freezing state of \( \varphi \).

The following result characterizes the relationship between equilibrium states and freezing states.

\begin{theorem}\label{freezing equivalent conditions}
Let \( (X, f) \) be a TDS with \( h_{top}(f) < +\infty \), and let \( \varphi, \psi \in C(X, \mathbb{R}) \) with \( \psi > 0 \). Let \( \mathcal{F} \) be a non-empty subset of \( \mathcal{M}(X, f) \). Then for all \( \beta \geqslant 1 \), \( \mathcal{F} \) is the set of equilibrium states for \( P_\psi(\beta \varphi) \) if and only if the following properties hold:
\begin{enumerate}
\item[(1)] \( \mathcal{F} \) is the set of equilibrium states of \( P_\psi(\varphi) \);
\item[(2)] For all \( \mu \in \mathcal{F} \),
\[
\frac{\int \varphi \, d\mu}{\int \psi \, d\mu} = \sup_{\nu \in \mathcal{M}(X, f)} \frac{\int \varphi \, d\nu}{\int \psi \, d\nu}.
\]
\end{enumerate}
\end{theorem}
\begin{proof}
First, assume that properties (1) and (2) hold.

Let \( \beta \geqslant 1 \), for every \( \nu \in \mathcal{M}(X, f) \backslash \mathcal{F} \) and every \( \mu \in \mathcal{F} \). By Theorem \ref{induced Variational Principle} and  (2), we have that
\begin{align*}
\frac{h_{\mu}(f) + \int \beta \varphi \, d\mu}{\int \psi \, d\mu}  = & \frac{h_{\mu}(f) + \int \varphi \, d\mu}{\int \psi \, d\mu} + (\beta - 1) \frac{\int \varphi \, d\mu}{\int \psi \, d\mu} \\
> &  \frac{h_{\nu}(f) + \int \varphi \, d\nu}{\int \psi \, d\nu} + (\beta - 1) \frac{\int \varphi \, d\nu}{\int \psi \, d\nu} \\
= & \frac{h_{\nu}(f) + \beta \int \varphi \, d\nu}{\int \psi \, d\nu}.
\end{align*}
This implies that  the set of equilibrium states for \( \beta \varphi \) and $\psi$ is contained in \( \mathcal{F} \).

For every \( \mu \in \mathcal{F} \) and \( \nu \in \mathcal{M}(X, f) \), one can prove in a similar fashion that
\[
\frac{h_{\mu}(f) + \int \beta \varphi \, d\mu}{\int \psi \, d\mu}
\geqslant \frac{h_{\nu}(f) + \beta \int \varphi \, d\nu}{\int \psi \, d\nu}.
\]
Therefore, \( \mathcal{F} \) is the set of equilibrium states for \( \beta \varphi \) and $\psi$.

Next, assume that for all \( \beta \geqslant 1 \), \( \mathcal{F} \) is the set of equilibrium states for \( \beta \varphi \) and $\psi$.  The first property  obviously holds by taking \( \beta = 1 \). Now, we prove the second property  holds. Take \( \nu \in \mathcal{M}(X, f) \) and \( \mu \in \mathcal{F} \). For every \( \beta \geqslant 1 \), we have that
\[
\begin{aligned}
\frac{h_{\mu}(f) + \int \beta \varphi \, d\mu}{\int \psi \, d\mu} \geqslant \frac{h_{\nu}(f) + \beta \int \varphi \, d\nu}{\int \psi \, d\nu}.
\end{aligned}
\]
Hence, we have
\[
\begin{aligned}
\frac{h_{\mu}(f)}{\beta \int \psi \, d\mu} + \frac{\int \varphi \, d\mu}{\int \psi \, d\mu} &\geqslant \frac{h_{\nu}(f)}{\beta \int \psi \, d\nu} + \frac{\int \varphi \, d\nu}{\int \psi \, d\nu}.
\end{aligned}
\]
Letting  \( \beta \to \infty \), we have that
\[
\begin{aligned}
\frac{\int \varphi \, d\mu}{\int \psi \, d\mu} \geqslant \frac{\int \varphi \, d\nu}{\int \psi \, d\nu}.
\end{aligned}
\]
This implies the desired result.
\end{proof}

Given $\varphi, \psi\in C(X, \mathbb{R})$ with $\psi>0$, let
\[
\operatorname{Max}^{\psi}(\varphi) := \sup_{\mu \in \mathcal{M}(X, f)} \frac{\int \varphi \, d\mu}{\int \psi \, d\mu}
\]
the above supremum is always attained by an invariant measure, since the map $\displaystyle \mu\mapsto \frac{\int \varphi \, d\mu}{\int \psi \, d\mu}$ is continuous and $\mathcal{M}(X, f)$ is weak$^*$ compact.
Furthermore, let
\[
M_{\text{max}}^{\psi}(\varphi) := \Big\{ \mu \in \mathcal{M}(X, f) : \frac{\int \varphi \, d\mu}{\int \psi \, d\mu} = \operatorname{Max}^{\psi}(\varphi) \Big\}
\]
and
\[
h_{\infty}^{\psi}(\varphi) := \sup \Big\{ \frac{h_{\mu}(f)}{\int \psi \, d\mu} : \mu \in M_{\text{max}}^{\psi}(\varphi) \Big\}.
\]

\begin{corollary}
Let \( (X, f) \) be a TDS with \( h_{top}(f) < +\infty \), and let \( \varphi, \psi \in C(X, \mathbb{R}) \) with \( \psi > 0 \). If \( \mu_{\infty} \) is the zero temperature limit of \( \varphi \), then \( \mu_{\infty} \in M_{\text{max}}^{\psi}(\varphi) \).
\end{corollary}

\begin{proof}
Since \( \mu_{\infty} \) is the zero temperature limit of \( \varphi \), there exists a sequence of invariant measures \( \left\{\mu_{\beta}\right\} \) that converges to \( \mu_{\infty} \) as \( \beta \to \infty \), where $\mu_\beta$ is the  equilibrium state for \( P_\psi(\beta \varphi) \).
Morevoer, for every \( \nu \in \mathcal{M}(X, f) \) and \( \beta>0 \), we have that
\[
\begin{aligned}
\frac{h_{\mu_{\beta}}(f) + \int \beta \varphi \, d\mu_{\beta}}{\int \psi \, d\mu_{\beta}} \geqslant \frac{h_{\nu}(f) + \beta \int \varphi \, d\nu}{\int \psi \, d\nu}.
\end{aligned}
\]
Hence, we have
\[
\begin{aligned}
\frac{h_{\mu_{\beta}}(f)}{\beta \int \psi \, d\mu_{\beta}} + \frac{\int \varphi \, d\mu_{\beta}}{\int \psi \, d\mu_{\beta}} \geqslant \frac{h_{\nu}(f)}{\beta \int \psi \, d\nu} + \frac{\int \varphi \, d\nu}{\int \psi \, d\nu}.
\end{aligned}
\]
Let \( \beta \to \infty \),  we have that
\[
\frac{\int \phi \, d\mu_{\infty}}{\int \psi \, d\mu_{\infty}} \geqslant \frac{\int \phi \, d\nu}{\int \psi \, d\nu}.
\]
The arbitrariness of $\nu\in \mathcal{M}(X, f)$ implies the desired result.
\end{proof}

\begin{corollary}
Let \( (X, f) \) be a TDS with \( h_{top}(f) < +\infty \), and let \(  \psi \in C(X, \mathbb{R}) \) with \( \psi > 0 \). Fix \( \mu \in \mathcal{M}(X, f) \), if there exists some \( \phi \in C(X, \mathbb{R}) \) such that \( \mu \) is the equilibrium state for \( P_\psi(\phi) \), and \( h_{\mu}(f) = 0 \), then \( \mu \) is the equilibrium state for \( P_\psi(\beta \phi) \)  for every \( \beta \geqslant 1 \).
\end{corollary}
\begin{proof}
By Theorem \ref{freezing equivalent conditions}, it suffices to prove that \( \mu \in M_{\text{max}}^{\psi}(\phi) \).
Since \( \mu \) is the equilibrium state for \( P_\psi(\phi) \), and \( h_{\mu}(f) = 0 \),   we have that
\[
\begin{aligned}
\frac{\int \phi \, d\mu}{\int \psi \, d\mu}=\frac{h_{\mu}(f) + \int \phi \, d\mu}{\int \psi \, d\mu}
\geqslant  \frac{h_{\nu}(f) + \int \phi \, d\nu}{\int \psi \, d\nu} \geqslant \frac{\int \phi \, d\nu}{\int \psi \, d\nu}\quad \forall \nu \in \mathcal{M}(X, f).
\end{aligned}
\]
 This implies that \( \mu \in M_{\text{max}}^{\psi}(\phi) \).
\end{proof}

\begin{proposition} \label{Necessary Conditions for Differentiability}
Let \( (X, f) \) be a TDS with \( h_{top}(f) < +\infty \), and let \( \varphi, \psi \in C(X, \mathbb{R}) \) with \( \psi > 0 \), \(\beta \in \mathbb{R} \). Put \( P_{\psi, \varphi}(\beta) = P_{\psi}(\beta \varphi) \), and assume that \( P_{\psi, \varphi} \) is differentiable at \( \beta_0 \) and that there exists at least one equilibrium state for \( P_\psi(\beta_0 \varphi) \). Then we have that
\[
\frac{\int \varphi \, d\mu}{\int \psi \, d\mu} = \frac{\partial P_{\psi, \varphi}}{\partial \beta} \left( \beta_0 \right),\quad \forall \mu\in \mathcal{M}_{\psi, \beta_0\varphi}(X, f).
\]
\end{proposition}
\begin{proof}
Given \( \beta > \beta_0 \) and $\mu\in \mathcal{M}_{\psi, \beta_0\varphi}(X, f)$, by Theorem  \ref{induced Variational Principle} we have that
\[
\frac{P_{\psi, \varphi}(\beta) - P_{\psi, \varphi}(\beta_0)}{\beta - \beta_0} \geqslant \frac{h_{\mu}(f) + \beta \int \varphi \, d\mu - h_{\mu}(f) - \beta_0 \int \varphi \, d\mu}{(\beta - \beta_0)\int \psi \, d\mu} =  \frac{\int \varphi \, d\mu}{\int \psi \, d\mu}.
\]
Taking \( \beta \to \beta_0^+ \), it follows that the right derivative is bounded from below by \( \frac{\int \varphi \, d\mu}{\int \psi \, d\mu} \).

Similarly, consider \( \beta < \beta_0 \), one can show  that the  left derivative is bounded from above by \( \frac{\int \varphi \, d\mu}{\int \psi \, d\mu} \). Since \( P_{\psi, \varphi} \) is differentiable at \( \beta_0 \), one can conclude that
\[
\frac{\int \phi \, d\mu}{\int \psi \, d\mu} = \frac{\partial P_{\psi, \varphi}}{\partial \beta} \left( \beta_0 \right).
\]
This completes the proof of the proposition.
\end{proof}

The following theorem gives the necessary and sufficient condition for the freezing state of a potential.
\begin{theorem}
Let \( (X, f) \) be a TDS with \( h_{top}(f) < +\infty \), and let \( \varphi, \psi \in C(X, \mathbb{R}) \) with \( \psi > 0 \). The potential function \( \varphi \) freezes at \( \beta_0 \) if and only if the following properties hold:
\begin{enumerate}
\item[(1)] \( P_{\psi, \varphi}(\beta)  = \beta \cdot \operatorname{Max}^{\psi}(\varphi) + h_{\infty}^{\psi}(\varphi) \) for every  \( \beta\ge \beta_0 \) ;
\item[(2)] there exists \( \beta_1 > \beta_0 \), such that \( \beta_1 \varphi \) has at least one equilibrium state.
\end{enumerate}
\end{theorem}
\begin{proof}
Assume that \( \varphi \) freezes at \( \beta_0 \), there exists $\mu\in \mathcal{M}(X,f)$ such that
\[
P_{\psi, \varphi}(\beta) = \frac{h_{\mu}(f) + \beta \int \varphi \, d\mu}{\int \psi \, d\mu},\quad \forall \beta \geqslant \beta_0.
\]
The second property clearly holds.

To show the first property,
we first prove that \( \frac{\int \varphi \, d\mu}{\int \psi \, d\mu} = \operatorname{Max}^{\psi}(\varphi) \). Indeed, for every \( \nu \in \mathcal{M}(X, f) \), by  Theorem \ref{induced Variational Principle} we have that
\[
\frac{h_{\mu}(f) + \int \beta \varphi \, d\mu}{\int \psi \, d\mu} \geqslant \frac{h_{\nu}(f) + \beta \int \varphi \, d\nu}{\int \psi \, d\nu},\quad \forall\beta \geqslant \beta_0.
\]
Dividing by \( \beta \) on both sides, and letting \( \beta \to \infty \), we have that
\[
\begin{aligned}
\frac{\int \varphi \, d\mu}{\int \psi \, d\mu} \geqslant \frac{\int \varphi \, d\nu}{\int \psi \, d\nu},
\end{aligned}
\]
since the topological entropy is finite. Furthermore, for each \( \nu \in M_{max}^{\psi}(\varphi) \), by Theorem \ref{induced Variational Principle} we have that
\[
\frac{h_{\mu}(f) + \int \beta \varphi \, d\mu}{\int \psi \, d\mu} \geqslant \frac{h_{\nu}(f) + \beta \int \varphi \, d\nu}{\int \psi \, d\nu}.
\]
Since \( \mu, \nu \in M_{max}^{\psi}(\varphi) \), we have that
\[
\frac{h_{\mu}(f)}{\int \psi \, d\mu} \geqslant \frac{h_{\nu}(f)}{\int \psi \, d\nu}.
\]
This yields that  \( \frac{h_{\mu}(f)}{\int \psi \, d\mu} = h_{\infty}^{\psi}(\varphi) \). Hence, we have that
\[
P_{\psi, \varphi}(\beta)  = \beta \cdot \operatorname{Max}^{\psi}(\varphi) + h_{\infty}^{\psi}(\varphi),\quad \forall\beta \geqslant \beta_0.
\]

Assume that  \( P_{\psi, \varphi}(\beta) = \beta \cdot \operatorname{Max}^{\psi}(\varphi) + h_{\infty}^{\psi}(\varphi) \) for every \( \beta \ge \beta_{0} \), and  \( \mu \) is an equilibrium state for \( P_\psi(\beta_1 \varphi) \). Note that  \( P_{\psi, \varphi}(\beta) \)  is differentiable at \( \beta_1 \), it follows from Proposition \ref{Necessary Conditions for Differentiability} that
\[
\frac{\int \varphi \, d\mu}{\int \psi \, d\mu} = \frac{\partial P_{\psi, \varphi}}{\partial \beta} \left( \beta_{1} \right) = \operatorname{Max}^{\psi}(\varphi).
\]
Hence, we have that
\[
\frac{h_{\mu}(f) + \beta_{1} \int \varphi \, d\mu}{\int \psi \, d\mu} = P_{\psi, \varphi}(\beta_{1}) =  \beta_{1} \cdot \operatorname{Max}^{\psi}(\varphi) + h_{\infty}^{\psi}(\varphi),
\]
which implies that $
\frac{h_{\mu}(f)}{\int \psi \, d\mu} = h_{\infty}^{\psi}(\varphi)$.
Thus, we have  that
\[
P_{\psi, \varphi}(\beta) = \frac{h_{\mu}(f) + \beta \int \varphi \, d\mu}{\int \psi \, d\mu},\quad \forall\beta \geqslant \beta_{0}.
\]
Therefore, \( \varphi \) freezes at \( \beta_{0} \).
\end{proof}

\section{High dimensional nonlinear topological pressure}\label{NTP}
This section first recalls the high dimensional nonlinear topological pressure defined in \cite{barreira2022higher}, and gives some basic properties of it. Moreover, following the approach as described in \cite{xing2015induced}, we define the high dimensional nonlinear induced topological pressure, and establish the variational principle for it.

\subsection{Properties of high dimensional nonlinear topological pressure} Let \((X, f)\) be a TDS,   $\Phi = (\varphi_{1},\varphi_2,\cdots,\varphi_{d})\in C(X, \mathbb{R})^d$ and   \(F: \mathbb{R}^{d} \rightarrow \mathbb{R}\)  a continuous function. For every \( \varepsilon > 0 \), let
\[
P_{n}^{F}(f, \Phi, \varepsilon) = \sup_{E} \sum_{x \in E} \exp \left[ n F \Big( \frac{S_{n} \varphi_{1}(x)}{n}, \ldots, \frac{S_{n} \varphi_{d}(x)}{n} \Big) \right],
\]
where the supremum is taken over all \( (n, \varepsilon) \)-separated sets.
The following quantity
\[
P^{F}(\Phi) = \lim_{\varepsilon \rightarrow 0} P^{F}(f, \Phi, \varepsilon)
\]
is called the \emph{nonlinear topological pressure} of \( \Phi \), where
\[
P^{F}(f, \Phi, \varepsilon) = \limsup_{n \rightarrow \infty} \frac{1}{n} \log P_{n}^{F}(f, \Phi, \varepsilon).
\]

In \cite{barreira2022higher}, the authors also give an equivalent definition of  nonlinear topological pressure  using  \( (n, \varepsilon) \)-spanning sets.
Let
\[
Q_{n}^{F}(f, \Phi, \varepsilon) = \inf_C \sum_{x \in C} \exp \left[ n F \Big( \frac{S_{n} \varphi_{1}(x)}{n}, \ldots, \frac{S_{n} \varphi_{d}(x)}{n} \Big) \right],
\]
where the infimum is  taken over all  \( (n, \varepsilon) \)-spanning sets of \( X \).
Let
\[
Q^{F}(f, \Phi, \varepsilon) = \limsup_{n \rightarrow \infty} \frac{1}{n} \log Q_{n}^{F}(f, \Phi, \varepsilon).
\]
Then one has that
\[
P^{F}(\Phi) = \lim_{\varepsilon \rightarrow 0} Q^{F}(f, \Phi, \varepsilon).
\]
Given  \( \Phi = \left\{ \varphi_{1}, \ldots, \varphi_{d} \right\} \in C(X,\mathbb{R})^d\), let
$\displaystyle \|\Phi\|_\infty = \max_{j \in \{1,2, \cdots, d\}} \left\|\varphi_{j}\right\|$, one can show that
 the map \( \Phi \mapsto P^{F}(\Phi) \) is continuous with respect to this norm.

\begin{definition}
Let \((X, f)\) be a TDS, and  $\Phi \in C(X, \mathbb{R})^d$. We say that \( (f, \Phi) \) has an abundance of ergodic measures, if for every \( \mu \in \mathcal{M}(X, f) \), \( h < h_{\mu}(f) \) and  \( \varepsilon > 0 \), there exists an ergodic measure \( \nu \in \mathcal{E}(X, f) \), such that \( h_{\nu}(f) > h \) and
\[
\left| \int{{{\varphi }_{i}}}d\nu-\int{{{\varphi }_{i}}}d\mu  \right| < \varepsilon, \quad  i = 1, \cdots , d.
\]
\end{definition}

In \cite{barreira2022higher}, Barreira and Holanda established the variational principle for the high dimensional nonlinear topological pressure as follows:

\begin{theorem} \label{high nonliner vp}
Let \((X, f)\) be a TDS and  $\Phi \in C(X, \mathbb{R})^d$, and let \(F: \mathbb{R}^{d} \rightarrow \mathbb{R}\) be continuous.
If either of the following conditions is satisfied:
\begin{enumerate}
\item[(1)] $(f, \Phi)$ has an abundance of ergodic measures;
\item[(2)] $F$ is a convex function,
\end{enumerate}
then
\[
P^{F}(\Phi) = \sup \Big\{ h_{\mu}(f) + F \Big( \int \Phi d \mu \Big): \mu \in \mathcal{M}(X,f)\Big\},
\]
where \( \displaystyle \int \Phi d \mu = \Big( \int \varphi_{1} d \mu, \cdots, \int \varphi_{d} d \mu \Big) \).
\end{theorem}

 By Theorem \ref{high nonliner vp}, it is easy to see that \(\displaystyle h_{\text {top }}(f) + \inf_{x \in \Phi(X)} F(x) \leqslant P^{F}(\Phi) \leqslant h_{\text {top }}(f) + \sup _{x \in \Phi(X)} F(x) \).
 The next result shows that  the high dimensional nonlinear topological pressure is a topologically conjugate invariant.

 \begin{theorem}
Let \((X, f)\) be a TDS and  $\Phi\in C(X, \mathbb{R})^d$, and let \(F: \mathbb{R}^{d} \rightarrow \mathbb{R}\) be continuous. The following properties hold:
\begin{enumerate}
\item[(i)] If \( f \) is a homeomorphism, then \( P^{F}(f^{-1}, \Phi) = P^{F}(f, \Phi) \); 
\item[(ii)] Let \( ( Y, g ) \) be a TDS and  \( \phi: X \rightarrow Y \)  a surjective continuous map such that \( \phi \circ f = g \circ \phi \). For every $\Phi \in C(Y, \mathbb{R})^d$, let \( \Phi \circ \phi = (\varphi_{1} \circ \phi, \cdots, \varphi_{d} \circ \phi) \), we have that \( P^{F}(g, \Phi) \leqslant P^{F}(f, \Phi \circ \phi) \). If \( \phi\) is a homeomorphism, then \( P^{F}(g, \Phi) = P^{F}(f, \Phi \circ \phi) \).
\end{enumerate}
\end{theorem}

\begin{proof}
(i) Given \( \varepsilon > 0 \) and \( n \in \mathbb{N} \), it is easy to see that  a subset $E\subset X$ is  $(n, \varepsilon)$-separated with respect to $f$ if and only if
 $f^{n-1} E$ is  $(n, \varepsilon)$-separated with respect to $f^{-1}$. Furthermore, one has that
\[
\sum_{x \in E} \exp \left[n F \left( \frac{S_{n} \Phi(x)}{n} \right) \right] = \sum_{y\in f^{n-1} E} \exp \left[n F \left( \frac{ \sum_{i=0}^{n-1}\Phi(f^{-i}y) }{n} \right) \right],
\]
where $S_n\Phi=(S_n\varphi_1,S_n\varphi_2,\cdots, S_n\varphi_d)$.
 This yields that \(P^{F}(f, \Phi, \varepsilon) = P^{F}(f^{-1}, \Phi, \varepsilon) \) and \( P^{F}(f^{-1}, \Phi) = P^{F}(f, \Phi) \).

(ii)  Given $\varepsilon>0$, there exists $0<\delta<\varepsilon$ such that $d_{X}(x_1, x_2)<\delta$ implies that  $d_{Y}(g(x_1), g(x_2))< \varepsilon$, here  $d_{X}$ and $d_{Y}$  denote the metrics on $X$ and $Y$ respectively.

Given \( \varepsilon > 0 \) and \( n \in \mathbb{N} \), if $C$ is an $(n, \delta)$-spanning set for $X$ with respect to $f$, then $\phi(C)$ is an $(n, \varepsilon)$-spanning set for $Y$ with respect to $g$. Therefore, for every $\Phi \in C(Y, \mathbb{R})^d$, we have that
\begin{align*}
&\sum_{x \in C} \exp \left[n F \left( \frac{\Phi(\phi x) + \Phi\left(\phi f x\right) + \cdots + \Phi\left(\phi f^{n-1} x\right)}{n} \right) \right] \\
&= \sum_{y \in \phi(C)} \exp \left[n F \left( \frac{\Phi(y) + \Phi\left(g y\right) + \cdots +\Phi\left(g^{n-1} y\right)}{n} \right) \right] \\
&\geq Q_{n}^{F}\left(g, \Phi, \varepsilon\right).
\end{align*}
Therefore, \( Q^{F}\left(f, \Phi \circ \phi, \delta\right) \geq Q^{F}\left(g, \Phi, \varepsilon\right) \) and \( P^{F}\left(f, \Phi \circ \phi \right) \geq P^{F}\left(g, \Phi\right) \).

If \( \phi \) is a homeomorphism, considering $g,f,\phi^{-1}, \Phi\circ\phi $ instead of $f, g, \phi, \Phi$, one has that $P^{F}\left(g, \Phi\right) \ge P^{F}\left(f, \Phi \circ \phi \right)$. This completes the proof of the theorem.
\end{proof}

\subsection{Definition of nonlinear induced topological pressure}
Following the approaches as described in \cite{xing2015induced} and \cite{barreira2022higher}, the high dimensional nonlinear induced topological pressure is defined as follows.

Let \((X, f)\) be a TDS, \(\psi \in C(X, \mathbb{R})\) with \(\psi > 0\),  and let \(F: \mathbb{R}^{d} \rightarrow \mathbb{R}\) be continuous.
For \(T > 0\), let
\[
S_{T} = \left\{n \in \mathbb{N} :    S_{n} \psi(x) \leqslant T < S_{n+1} \psi(x) \text{ for some } x \in X\right\}.
\]
For \( n \in S_{T} \), let
\[
X_{n} = \left\{x \in X : S_{n} \psi(x) \leqslant T < S_{n+1} \psi(x)\right\}.
\]
Given $\Phi\in C\left(X, \mathbb{R}\right)^d $ and \( \varepsilon > 0 \), define
\[
\begin{aligned}
Q_{\psi, T}^{F}(f, \Phi, \varepsilon) =
\inf \Big \{ &\sum_{n \in S_{T}} \sum_{x \in F_{n}}
\exp \left( n F \left( \frac{S_{n} \Phi(x)}{n} \right) \right):\\
&  F_{n} \text{ is an } (n, \varepsilon) \text{-spanning set of } X_{n}, \ n \in S_{T} \Big\}.
\end{aligned}
\]
The nonlinear induced topological pressure
of $\Phi$ is then defined by
\begin{align}\label{5*}
P_{\psi}^{F}(f,\Phi) = \lim _{\varepsilon \rightarrow 0} Q_{\psi}^{F}(f, \Phi, \varepsilon),
\end{align}
where
$\displaystyle
Q_{\psi}^{F}(f, \Phi, \varepsilon) = \limsup_{T \rightarrow \infty} \frac{1}{T} \log Q_{\psi, T}^{F}(f, \Phi, \varepsilon)
$. If there is no confusion caused, we simply write $P_{\psi}^{F}(f,\Phi)$  as $P_{\psi}^{F}(\Phi)$.

\begin{remark}(i) Let \( m = \inf \psi \), \( M = \sup \psi \). Fix \( T>0 \),  we have that \( n \leqslant \frac{T}{m} \) for each \( n \in S_{T}\) , i.e., \( S_{T} \) is a finite set; (ii) If \( 0 < \varepsilon_1 < \varepsilon_2 \), then \( Q_{\psi, T}\left(f, \Phi, \varepsilon_1\right) \geqslant Q_{\psi, T}\left(f, \Phi, \varepsilon_2\right) \), which implies the existence of  the limit in \eqref{5*}  and \( P_{\psi}^F(\Phi) > -\infty \).
\end{remark}

The following theorem gives an equivalent definition of the nonlinear induced topological pressure  using separated sets.
\begin{theorem}\label{equiv-sep}
Let \((X, f)\) be a TDS, \(\psi \in C(X, \mathbb{R})\) with \(\psi > 0\),  and let \(F: \mathbb{R}^{d} \rightarrow \mathbb{R}\) be continuous. Given $\Phi\in C\left(X, \mathbb{R}\right)^d $,
\( T > 0 \) and \( \varepsilon > 0 \), let
\[
\begin{aligned}
P_{\psi, T}^F(f, \Phi, \varepsilon) = \sup \Big\{ &\sum_{n \in S_T} \sum_{x \in F_n}
\exp \left( n F \left( \frac{S_n \Phi(x)}{n} \right) \right):\\
 &F_n \text{ is an } (n, \varepsilon)\text{-separated set of } X_n, \ n \in S_T \Big\}.
\end{aligned}
\]
Then
\[
P_{\psi}^{F}(\Phi) = \lim _{\varepsilon \rightarrow 0} P_\psi^F(f, \Phi, \varepsilon),
\]
where $\displaystyle
P_{\psi}^F(f, \Phi, \varepsilon) = \limsup_{T \rightarrow \infty} \frac{1}{T} \log P_{\psi, T}^{F}(f, \Phi, \varepsilon)$.
\end{theorem}
\begin{proof}
Given \( T > 0 \) and  \( n \in S_{T} \), let \( E_{n} \) be a maximal \((n, \varepsilon)\)-separated set of \( X_{n} \), then \( E_{n} \) is an \((n, \varepsilon)\)-spanning set of \( X_{n} \). This yields that
$
Q_{\psi, T}^{F}(f, \Phi, \varepsilon) \leqslant P_{\psi, T}^{F}(f, \Phi, \varepsilon).
$
Hence, we have
\[
P_{\psi}^{F}(\Phi) \leqslant \lim _{\varepsilon \rightarrow 0} \limsup_{T \rightarrow \infty} \frac{1}{T} \log P_{\psi, T}^{F}(f, \Phi, \varepsilon).
\]

  Next, we will show that
\[
P_{\psi}^{F}(\Phi) \geqslant \lim _{\varepsilon \rightarrow 0} \limsup_{T \rightarrow \infty} \frac{1}{T} \log P_{\psi, T}^{F}(f, \Phi, \varepsilon).
\]
Since \(\Phi\)  and  \(F\)  are continuous,  for every \(\varepsilon > 0\), there exists \(\bar{\delta},\delta  > 0\) such that
\[
 \varrho(\alpha, \beta) < \bar{\delta} \implies |F(\alpha) - F(\beta)| < \varepsilon, \quad \forall \alpha, \beta \in \Phi(X)
\]
and
\[
 d(x, y) \leqslant \frac{\delta}{2} \implies |\varphi_{i}(x) - \varphi_{i}(y)| < \bar{\delta}\,\, (i=1,2, \cdots, d),\quad\forall x, y\in X
\]
where $\varrho$ is the maximum norm on $\mathbb{R}^d$.

Given \( T > 0 \) and \(n \in S_T\), let \( E_n \) be an \((n, \delta)\)-separated set of \( X_n \), and \( F_n \) be an \((n, \frac{\delta}{2})\)-spanning set of \( X_n \). Define $\phi: E_{n} \rightarrow F_{n}$ by choosing,  for each $x \in E_{n}$, some point $\phi(x)\in F_{n}$ with $d_n(x, \phi(x)) \leqslant \frac{\delta}{2}$. It is easy to see that $\phi$ is injective.
Thus, we have that
$$
\begin{aligned}
&\sum_{y \in F_n} \exp \left( nF \left( \frac{S_n \Phi(y)}{n} \right) \right) \\
&\geqslant \sum_{x \in E_n} \exp \left( nF \left( \frac{S_n \Phi(\phi(x))}{n} \right) \right) \\
&= \sum_{x \in E_n} \exp \left( nF \left( \frac{S_n \Phi(x)}{n} \right) \right)
\exp \left( nF \left( \frac{S_n \Phi(\phi(x))}{n} \right) - nF \left( \frac{S_n \Phi(x)}{n} \right) \right) \\
&\geqslant \min_{x \in E_n} \exp \left( nF \left( \frac{S_n \Phi(\phi(x))}{n} \right)
- nF \left( \frac{S_n \Phi(x)}{n} \right) \right)
\sum_{x \in E_n} \exp \left( nF \left( \frac{S_n \Phi(x)}{n} \right) \right).
\end{aligned}
$$
Since $d_n(x, \phi(x)) \leqslant \frac{\delta}{2}$ for every \( x \in E_n \),
for each $i = 1, 2,\cdots,d$,  we have that
\[
\Big|\varphi_i(f^jx)-\varphi_i(f^j\phi(x))\Big|< \bar{\delta}, \quad \forall j= 1, 2,\cdots, n-1,
\]
which implies that
\[
\Big|\frac{S_n \varphi_i(\phi(x))}{n}-\frac{S_n \varphi_i(x)}{n}\Big| < \bar{\delta}, \quad \forall i = 1, \cdots, d.
\]
Hence, we have that
\[
\varrho\Big( \frac{S_n \Phi(\phi(x))}{n}, \frac{S_n \Phi(x)}{n}\Big) < \bar{\delta}.
\]
This yields that
\[
\Big| nF\Big(\frac{S_n \Phi(x)}{n}\Big)- nF\Big(\frac{S_n \Phi(\phi(x))}{n}\Big)\Big| < n\varepsilon< \frac{T\varepsilon}{m},\quad \forall x \in E_n
\]
since \( n \leqslant \frac{T}{m} \).
Hence, we have that
\begin{align*}
\sum_{y \in F_n} \exp\Big(nF\Big(\frac{S_n \Phi(y)}{n}\Big)\Big)
\geqslant e^{-\frac{T}{m}\varepsilon} \sum_{x \in E_n} \exp\Big(nF\Big(\frac{S_n \Phi(x)}{n}\Big)\Big).
\end{align*}
Consequently, one has that
\[
\begin{aligned}
P_F^{\psi}(\Phi) &= \lim_{\delta \to 0} \limsup_{T \to \infty} \frac{1}{T} \log Q_{\psi, T}^F(f,\Phi, \frac{\delta}{2}) \\
&\geqslant \lim_{\delta \to 0} \limsup_{T \to \infty} \frac{1}{T} \log P_{\psi, T}^F(f,\Phi, \delta) - \frac{\varepsilon}{m}.
\end{aligned}
\]
Since  \(\varepsilon \) is chosen arbitrarily, the desired result follows immediately.
\end{proof}

\begin{remark} As in Section \ref{equivalent Top-Pressure}, one can define the nonlinear induced topological pressure via open covers. Moreover, it can be shown in a similar fashion that Theorem \ref{open cover et}  is equivalent to the one defined by spanning sets.
\end{remark}

The following theorem shows that the nonlinear induced topological pressure is a topologically conjugate invariant.

\begin{theorem}
Let \( f: X \rightarrow X  \) be a continuous map of compact metric space \( \left(X, d_{1}\right) \), and let \( h: Y \rightarrow Y  \) be a continuous map of compact metric space \( \left(Y, d_{2}\right) \). Assume that \( F: \mathbb{R}^{d} \rightarrow \mathbb{R} \) is continuous, and  \( g: X \rightarrow Y \) is a  homeomorphism  such that \( g \circ f = h \circ g \). Let \(\psi \in C(Y, \mathbb{R})\) with \(\psi > 0\) and $
\Phi\in C\left(Y, \mathbb{R}\right)^d$,
 we have that
\[
P_{\psi}^{F}\left(h, \Phi\right) = P_{\psi \circ g}^{F}\left(f, \Phi \circ g\right),\]
  where  $\Phi \circ g = (\varphi_{1} \circ g, \cdots, \varphi_{d} \circ g)$.

\end{theorem}

\begin{proof}
For every $T>0$, let
$$
S^{X}_{T} := \left\{n \in \mathbb{N} : \exists x \in X \text{ such that } S_{n} \psi \circ g(x) \leqslant T < S_{n+1} \psi \circ g(x)\right\},
$$
and
$$
S^{Y}_{T} := \left\{n \in \mathbb{N} : \exists y \in X \text{ such that } S_{n} \psi(y) \leqslant T < S_{n+1} \psi(y)\right\}.
$$
If  \( \psi(y)+\psi(h(y))+\cdots+\psi(h^{n-1}(y)) \leqslant T < \psi(y)+\psi(h(y))+\cdots+\psi(h^{n}(y)) \) for some \( y \in Y \),  then  \( \psi(g(x))+\psi(g(f(x)))+...+\psi(g(f^{n-1}(x))) \leqslant T < \psi(g(x))+\psi(g(f(x)))+...+\psi(g(f^{n}(x))) \) by letting  \( x=g^{-1}(y) \in X \).  This implies that \( S^{Y}_{T} = S^{X}_{T} \). Let \( S_{T} \) denote the common set.

 Given $\varepsilon>0$, there exists $\delta\in (0,\varepsilon)$ such that  $d_{2}(g(x_1), g(x_2))< \varepsilon$  whenever $d_{1}(x_1, x_2)<\delta$. For every \( \varepsilon > 0 \), \( T >0 \) and \( n \in S_{T} \), if $F_n$ is an $(n, \delta)$-spanning set for $X_n$ with respect to $f$, then $g F_n$ is an $(n, \varepsilon)$-spanning set for $Y_n$ with respect to $h$. Therefore,  we have that
\begin{align*}
&\sum_{n \in S_{T}}\sum_{x \in F_n} \exp \Big[n F \Big( \frac{\Phi(g x) + \Phi(g f x) + \cdots + \Phi(g f^{n-1} x)}{n} \Big) \Big] \\
&= \sum_{n \in S_{T}}\sum_{y \in g F_n} \exp \Big[n F \Big( \frac{\Phi(y) + \Phi(h y) + \cdots +\Phi(h^{n-1} y)}{n} \Big) \Big] \\
&\geqslant Q_{\psi,T}^{F}(h, \Phi, \varepsilon).
\end{align*}
Hence, \(  Q_{\psi \circ g}^{F}(f, \Phi \circ g, \delta) \geqslant Q_{\psi}^{F}(h, \Phi, \varepsilon) \) and \( P_{\psi \circ g}^{F}(f, \Phi \circ g) \geqslant P_{\psi}^{F}(h, \Phi) \).

Since \( g \) is a homeomorphism, the reverse inequality \( P_{\psi \circ g }^{F}(f, \Phi \circ g) \leqslant P_{\psi}^{F}(h, \Phi) \) can be proven in a similar fashion.
\end{proof}

\subsection{Variational principle for nonlinear induced topological pressure}
This section gives the variational principle for high dimensional nonlinear induced topological pressure. The key ingredient of the proof is the study of the relation between the induced topological pressure and the nonlinear induced pressure.

Let $(X,f)$, $\psi$ be as before, for \( T > 0 \), let
$
G_T = \{n \in \mathbb{N} : \exists x \in X \text{ such that } S_n \psi(x) > T \},
$
and let
$
Y_n = \{x \in X : S_n \psi(x) > T\}
$
for \( n \in G_T \).
\begin{theorem} \label{induce nonliner p thm}
Let \((X, f)\) be a TDS, \(\psi \in C(X, \mathbb{R})\) with \(\psi > 0\),  and let \(F: \mathbb{R}^{d} \rightarrow \mathbb{R}\) be continuous.
Given $\Phi\in C(X,\mathbb{R})^d$, $\phi\in C(X,\mathbb{R})$ and \( \varepsilon > 0 \), let
\begin{align*}
R_{\psi, T}^F( \Phi, \phi, \varepsilon)
= \sup \Big\{&\sum_{n \in G_{T}} \sum_{x \in E_{n}^{\prime}}
\exp \Big[ n F\Big(\frac{S_n \Phi(x)}{n}\Big)
+ S_{n} \phi(x)\Big]:  \\
 & E_{n}^{\prime} \text{ is an } (n, \varepsilon)\text{-separated set of } Y_{n}, \ n \in G_{T} \Big\}.
\end{align*}
Then, we have that
\[
P_\psi^F(\Phi) = \inf\Big\{\beta \in \mathbb{R} : \lim_{\varepsilon \to 0} \limsup_{T \to \infty} R_{\psi, T}^F(  \Phi, -\beta \psi, \varepsilon) < \infty \Big\}.
\]
Here we take the convention that \(\inf \emptyset = \infty\).
\end{theorem}
\begin{proof}
For \( n \in \mathbb{N}, x \in X \), let \( m_n(x) \) denote the unique positive integer such that
\begin{equation}\label{***}
(m_n(x) - 1) \| \psi \|< S_n \psi(x) \leqslant m_n(x) \| \psi \|.
\end{equation}
Hence, one has that
\[
\begin{aligned}
\exp(-\beta m_n(x) \| \psi \|) \exp(-|\beta| \| \psi \|)
&\leqslant \exp(-\beta S_n \psi(x)) \\
&\leqslant \exp(-\beta m_n(x) \| \psi \|) \exp(|\beta| \| \psi \|)
\end{aligned}
\]
for every $x\in X$.
Given a family of functions \(\xi_T = \{\xi_n : X \to \mathbb{R} \mid n \in G_T\}\), for every \( \varepsilon > 0 \), define
\[
\begin{aligned}
R_{\psi, T}^F( \Phi, \xi_T, \varepsilon)
= \sup \Big\{ & \sum_{n \in G_T} \sum_{x \in E_n^{\prime}}
\exp\Big[ n F\Big(\frac{S_n \Phi(x)}{n}\Big) + \xi_n(x) \Big]:
 \\
& E_n^{\prime} \text{ is an  } (n, \varepsilon)\text{-separated set of } Y_n, \ n \in G_T \Big\}.
\end{aligned}
\]
This yields that
\begin{align*}
&\exp(-|\beta| \|\psi\|) R^{F}_{\psi, T}( \Phi, \left\{-\beta m_{n}\|\psi\|\right\}_{n \in G_{T}}, \varepsilon )\\
&\leqslant R^{F}_{\psi, T}( \Phi, -\beta \psi, \varepsilon) \\
&\leqslant \exp(|\beta| \|\psi\|) R^{F}_{\psi, T}( \Phi, \left\{-\beta m_{n}\|\psi\|\right\}_{n \in G_{T}}, \varepsilon).
\end{align*}
Consequently, one can conclude that
\[
\lim_{\varepsilon \to 0} \limsup_{T \to \infty} R^{F}_{\psi, T}( \Phi, -\beta \psi, \varepsilon) < \infty
\]
if and only if
 \[\lim_{\varepsilon \to 0} \limsup_{T \to \infty} R^{F}_{\psi, T}( \Phi, \left\{-\beta m_{n}\|\psi\|\right\}_{n \in G_{T}}, \varepsilon) < \infty.
\]
Thus, it suffices to show that
\[
P^{F}_{\psi}(\Phi) = \inf\Big\{\beta \in \mathbb{R} : \lim_{\varepsilon \to 0} \limsup_{T \to \infty} R^{F}_{\psi, T}(\Phi, \{-\beta m_{n}\|\psi\|\}_{n \in G_{T}}, \varepsilon) < \infty \Big\}.
\]

By Theorem \ref{equiv-sep}, for every \( \delta > 0 \),  \( \beta \in \mathbb{R} \) with \( \beta < P_{\psi}^{F}(\Phi) - \delta \), there exists \( \varepsilon_0 > 0 \), such that for every \( \varepsilon \in (0, \varepsilon_0) \), we have

\[
\beta + \delta < \limsup_{T \to \infty} \frac{1}{T} \log P^{F}_{\psi, T}(f, \Phi, \varepsilon) \leqslant P^{F}_{\psi}(\Phi).
\]
Therefore, for every \( T > 0 \) , there exists a sequence of numbers \( \{T_{j}\}_{j \in \mathbb{N}} \) such that for each $j \in \mathbb{N}$, the following properties hold:
\begin{enumerate}
\item[(i)]$T_{j+1} - T_j > 2\|\psi\|$ and $T_j - \|\psi\| > T$;
\item[(ii)]$\frac{1}{T_j} \log P^{F}_{\psi, T_j}(f,\Phi, \varepsilon) > \beta + \frac{\delta}{2}$.
\end{enumerate}
Fix $j \in \mathbb{N}$, for all $n \in S_{T_j}$, there exists an $(n, \varepsilon)$-separated set $E_n $ of $X_n$ such that
\begin{equation}\label{**}
\sum_{n \in S_{T_j}} \sum_{x \in E_n}
\exp\Big(nF\Big(\frac{S_n\Phi(x)}{n}\Big)\Big) \geqslant
\exp\Big(T_j(\beta + \frac{\delta}{2})\Big).
\end{equation}
Note that for every \( j \in \mathbb{N} \),  \( n \in S_{T_j} \) and  \( x \in E_n \),
$
T_j - \|\psi\| \leqslant S_n \psi(x) \leqslant T_j,
$
this implies that  \( S_{T_j} \cap S_{T_i} = \emptyset \)  for every \( i \neq j \). This together with \eqref{***} yield that
\[
-2\|\psi\| + T_j \leqslant S_n \psi(x) \leqslant  \|\psi\| m_n(x) \leqslant S_n \psi(x)+\|\psi\| \leqslant 2\|\psi\| + T_j.
\]
Hence,  we have that
\begin{equation}\label{*}
-\beta\|\psi\| m_n(x) \geqslant -2|\beta|\|\psi\| - \beta T_j, \quad \forall \beta \in \mathbb{R}.
\end{equation}
 By \eqref{**} and  \eqref{*}, we have that
\[
\begin{aligned}
R^{F}_{\psi, T}&\left(\Phi, \left\{-\beta m_n\|\psi\|\right\}_{n \in G_T}, \varepsilon\right) \\
&\geq \sum_{j \in \mathbb{N}, T_j - \|\psi\| > T} \sum_{n \in S_{T_j}} \sum_{x \in E_n}
\exp\Big(n F\Big(\frac{S_n \Phi(x)}{n}\Big) - \beta m_n(x)\|\psi\|\Big) \\
&\geqslant \exp\left(-2 |\beta| \|\psi\|\right)
\sum_{j \in \mathbb{N}, T_j - \|\psi\| > T}
\exp\Big(T_j (\beta + \frac{\delta}{2}) - \beta T_j\Big) \\
&= \infty.
\end{aligned}
\]
The above argument is not only valid for $P^{F}_{\psi}(\Phi)\in \mathbb{R}$, but also for $P^{F}_{\psi}(\Phi)=\infty$, in which case the previous inequality holds for
every $\beta\in\mathbb{R}$.
The arbitrariness of $\beta, \delta$ implies that
\[
P^{F}_{\psi}(\Phi) \leq \inf\Big\{\beta \in \mathbb{R} : \lim_{\varepsilon \to 0} \limsup_{T \to \infty} R^{F}_{\psi, T}(\Phi, \{-\beta m_{n}\|\psi\|\}_{n \in G_{T}}, \varepsilon) < \infty \Big\}.
\]

To prove the reverse inequality, given \( \delta > 0 \), we consider the case $P^{F}_{\psi}(\Phi)\in \mathbb{R}$ and show that
\[
\lim_{\varepsilon \to 0} \limsup_{T \to \infty} R_{\psi, T}^{F}( \Phi, \{-(P_{\psi}^{F}(\Phi)+\delta) m_{n}\|\psi\|\}_{n \in G_{T}} , \varepsilon ) < \infty.
\]
 By Theorem \ref{equiv-sep},  there exists \( \varepsilon_0 > 0 \) such that for every \( \varepsilon \in (0, \varepsilon_0) \), we have that
\[
\limsup_{T \to \infty} \frac{1}{T} \log P_{\psi, T}^F(f, \Phi, \varepsilon) < P_\psi^F(\Phi) + \frac{\delta}{2}.
\]
Fix such an $\varepsilon>0$, there exists \( \ell_0 \in \mathbb{N} \) such that for every positive
integer \( \ell > \ell_0 \), we have that
\begin{equation}\label{xx}
P_{\psi, \ell m}^{F}(f, \Phi, \varepsilon) < \exp\Big(\ell m \Big(P_\psi^F(\Phi) + \frac{2}{3} \delta\Big)\Big).
\end{equation}
For every \( n \in S_{\ell m}\) and \(x \in E_n \), by \eqref{***} and the fact that \( \ell m-\|\psi\| \leqslant S_n \psi(x) \leqslant \ell m \), we have that
\[
|\|\psi\| m_n(x) - \ell m| \leqslant |\|\psi\| m_n(x)-S_n \psi(x)|+|S_n \psi(x)-\ell m| \leqslant 2 \|\psi\|.
\]
Thus, one can conclude that
\begin{equation}\label{X}
 -(P_\psi^F(\Phi) + \delta) ||\psi|| m_n(x) \leqslant 2 |P_\psi^F(\Phi) + \delta| \cdot ||\psi|| - (P_\psi^F(\Phi) + \delta) \ell m.
\end{equation}
For sufficiently large $T > 0$, take $n \in G_T$ and let $E_n'$ be an $(n, \varepsilon)$-separated set for $Y_n$. For each $x \in E_n'$, there exists a unique $\ell \in \mathbb{N}$ with $\ell > \ell_0$ such that
\[
(\ell - 1)m < S_n \psi(x) \leqslant \ell m.
\]
Consequently, one has that
\[
\ell m > T \quad \text{and} \quad S_{n+1} \psi(x) > (\ell-1)m + m = \ell m.
\]
This implies that $(\ell - 1)m < S_r \psi(x) \leqslant \ell m$ if and only if  $r = n$.
Therefore, by \eqref{X} and \eqref{xx}, we have
\begin{align*}
& R_{\psi, T}^{F}( \Phi, \{ -(P_\psi^F(\Phi) + \delta) ||\psi|| m_n \}_{n \in G_T}, \varepsilon) \\
& \leqslant \sum_{\ell \geqslant \ell_0} \sup \Big\{
\sum_{n \in S_{\ell m}} \sum_{x \in E_n}
\exp\Big[ nF\Big( \frac{S_n \Phi(x)}{n}  \Big)
- ( P_\psi^F(\Phi) + \delta ) \|\psi\| m_n(x) \Big]: \\
& \qquad \qquad \qquad \qquad\qquad\qquad \qquad E_n \text{ is an } (n, \varepsilon)\text{-separated set of } X_n, n \in S_{\ell m}
\Big\} \\
& \leqslant \sum_{\ell \geqslant \ell_0} \sup \Big\{
\sum_{n \in S_{\ell m}} \sum_{x \in E_n}
\exp\Big( nF\Big( \frac{S_n \Phi(x)}{n} \Big) \Big) \cdot
\exp\Big( 2 |P_\psi^F(\phi) + \delta| \cdot ||\psi|| \Big) \\
& \qquad \qquad \cdot
\exp( -(P_\psi^F(\phi) + \delta) \ell m ): E_n \text{ is an } (n, \varepsilon)\text{-separated set of } X_n, n \in S_{\ell m} \Big\}\\
& \leqslant \exp\Big(2 |P_\psi^F(\Phi) + \delta| \cdot ||\psi||\Big)
\sum_{\ell \geqslant \ell_0} \exp\Big(-\ell m (P_\psi^F(\Phi) + \delta)\Big) P_{\psi, \ell m}^{F}( f,\Phi, \varepsilon) \\
& \leqslant \exp\Big(2 |P_\psi^F(\Phi) + \delta|\cdot||\psi||\Big) \sum_{\ell \geqslant \ell_0} \exp\Big(-\ell m (P_\psi^F(\Phi) + \delta)\Big) \exp\Big(\ell m (P_\psi^F(\Phi) + \frac{2}{3} \delta)\Big) \\
&  < \infty.
\end{align*}
Hence,
\[
\lim_{\varepsilon \to 0} \limsup_{T \to \infty} R_{\psi, T}^{F}(\Phi,  \left\{ -\left( P_\psi^F(\Phi) + \delta\right) ||\psi|| m_n(x) \right\}_{n \in G_T}, \varepsilon) < \infty.
\]
The arbitrariness of $\delta$ implies the desired result. The case of $P_\psi^F(\Phi)=\infty$ can be proven in a similar fashion.
\end{proof}

\begin{corollary} \label{induce nonliner p cor1}
Let \((X, f)\) be a TDS, \(\psi \in C(X, \mathbb{R})\) with \(\psi > 0\),  and let \(F: \mathbb{R}^{d} \rightarrow \mathbb{R}\) be continuous.  Given
$
\Phi \in C (X, \mathbb{R})^d$, let
$
\Psi(x) = (\Phi(x), \psi(x)) \in C (X, \mathbb{R})^{d+1}.
$
For each \( \beta \in \mathbb{R} \), define $G_{\beta}: \mathbb{R}^{d+1} \to \mathbb{R}$ as  $(a, b) \mapsto F(a) - \beta b$, where  $a \in \mathbb{R}^d$.
Then, we have that
\[
P_\psi^F(\Phi) \geqslant \inf\Big\{\beta \in \mathbb{R} : P^{G_{\beta}}(\Psi) \leqslant 0 \Big\}
\]
where $P^{G_{\beta}}(\Psi)$ is the nonlinear topological pressure of $\Psi$.
\end{corollary}
\begin{proof}
 Take \( \displaystyle \beta \in \Big\{\beta \in \mathbb{R} \mid \lim _{\varepsilon \rightarrow 0} \lim _{T \rightarrow \infty} R_{\psi, T}^{F}( \Phi, -\beta \psi, \varepsilon)<\infty\Big\} \),
and let
\[
\lim_{\varepsilon \to 0} \lim_{T \to \infty} R_{\psi, T}^{F}( \Phi, -\beta \psi, \varepsilon) = a.
\]
Fix a small \( \varepsilon > 0 \), there exists \( T_0 \) such that for every \( T > T_0 \),
\[
R_{\psi,T}^{F}(\Phi, -\beta \psi, \varepsilon) < a + 2.
\]
For sufficiently large \( n \in \mathbb{N} \),  we have \( S_n \psi(x) > T \) for every \( x \in X \). Given such \( n \in G_T \), let \( E_n^{'} \) be an \( (n, \varepsilon) \)-separated set of \( Y_n = X \),
hence
\begin{align*}
&\sum_{x \in E_n^{'}}  \exp\Big(nF\Big(\frac{S_n \Phi(x)}{n}\Big) - \beta S_n \psi(x)\Big) \\
& < \sum_{n \in G_T} \sum_{x \in E_n^{'}} \exp\Big(nF\Big(\frac{S_n \Phi(x)}{n}\Big) - \beta S_n \psi(x)\Big) \\
& < a + 2.
\end{align*}
This yields that
\begin{align*}
P^{G_{\beta}}(\Psi) &= \lim _{\varepsilon \rightarrow 0} \limsup _{n \rightarrow \infty}
\frac{1}{n} \log \sup \Big\{
\sum_{x \in E_n}  \exp\Big(nF\Big(\frac{S_n \Phi(x)}{n}\Big) - \beta S_n \psi(x)\Big) : \\
& \quad\quad\quad\quad\quad\quad\quad\quad\quad\quad\quad\quad\quad\quad  E_n \text{ is an } (n, \varepsilon)\text{-separated set of } X
\Big\} \leqslant 0.
\end{align*}
Therefore, by Theorem \ref{induce nonliner p thm}, we have
\[
\Big\{\beta \in \mathbb{R} : \lim _{\varepsilon \rightarrow 0} \lim _{T \rightarrow \infty} R_{\psi, T}^{F}( \Phi, -\beta \psi, \varepsilon)<\infty\Big\} \subset \Big\{\beta \in \mathbb{R} : P^{G_{\beta}}(\Psi) \leqslant 0 \Big\}.
\]
Consequently, one has that
\[
\begin{aligned}
P^{F}_{\psi}(\Phi) &=\inf \Big\{\beta \in \mathbb{R} : \lim _{\varepsilon \rightarrow 0} \lim _{T \rightarrow \infty} R_{\psi, T}^{F}( \Phi, -\beta \psi, \varepsilon)<\infty\Big\}\\
&\geqslant \inf\Big\{\beta \in \mathbb{R} : P^{G_{\beta}}(\Psi) \leqslant 0 \Big\}.
\end{aligned}
\]
This completes the proof of the corollary.
\end{proof}

Let \((X, f)\), $F$, $\psi$, $\Phi$, $\Psi$ and $G_{\beta}$ be the same as in Corollary \ref{induce nonliner p cor1}, the following result hold.

\begin{corollary} \label{induce nonliner p cor2}
If either of the following two conditions is satisfied:
\begin{enumerate}
\item[(1)] $(f, \Psi)$ has an abundance of ergodic measures;
\item[(2)] $F$ is a convex function,
\end{enumerate}
then we have that
\[
P^{F}_{\psi}(\Phi) = \inf\Big\{\beta \in \mathbb{R} : P^{G_{\beta}}(\Psi) \leqslant 0 \Big\}
= \sup\Big\{\beta \in \mathbb{R} : P^{G_{\beta}}(\Psi) \geqslant 0 \Big\}.
\]
\end{corollary}
\begin{proof}
It is easy to see that $G_\beta$ is a convex function for each $\beta\in \mathbb{R}$, since $F$ is a convex function.

If \( h_{\text {top }}(f) = +\infty \), by Theorem \ref{high nonliner vp}, one can see that $P^{G_{\beta}}(\Psi)=+\infty$ for each $\beta\in \mathbb{R}$. Hence, we have that
\[
P^{F}_{\psi}(\Phi) = \inf\left\{\beta \in \mathbb{R} : P^{G_{\beta}}(\Psi) \leqslant 0 \right\}
= \sup\left\{\beta \in \mathbb{R} : P^{G_{\beta}}(\Psi) \geqslant 0 \right\}=+\infty.
\]
Here we make the convention that $\inf \emptyset =+\infty$.

In the following, we assume that \( h_{\text {top }}(f) < +\infty \).
By Theorem \ref{high nonliner vp}, for every \( \beta \in \mathbb{R} \), we have
\begin{align*}
P^{G_{\beta}}(\Psi) & = \sup \Big\{h_{\nu}(f) + G_{\beta}\Big(\int \Psi d \nu\Big) : \nu \in \mathcal{M}(X, f)\Big\} \\
& = \sup \Big\{h_{\nu}(f) + F\Big(\int \Phi d\nu \Big) - \beta \int \psi d \nu : \nu \in \mathcal{M}(X, f)\Big\}\\
&<+\infty.
\end{align*}
Given \( \beta_1, \beta_2 \in \mathbb{R} \) with \( \beta_1 < \beta_2 \). For each \( 0 < \varepsilon \leqslant \frac{m(\beta_2 - \beta_1)}{2} \) (recall that $m=\inf \psi$), there exists \( \mu \in \mathcal{M}(X, f) \) such that
\[
\begin{aligned}
& \sup \Big\{  h_{\nu}(f) + F\Big(\int \Phi d\nu\Big) - \beta_2 \int \psi d \nu : \nu \in \mathcal{M}(X, f) \Big\} \\
& < h_{\mu}(f) + F\Big(\int \Phi d \mu\Big) - \beta_2 \int \psi d \mu + \varepsilon \\
& \le h_{\mu}(f) + F\Big(\int \Phi d \mu\Big) - (\beta_2 - \beta_1)\Big(\int \psi d \mu - \frac{m}{2}\Big) - \beta_1 \int \psi d \mu \\
& \leqslant \sup \Big\{ h_{\mu}(f) + F\Big(\int \Phi d \mu\Big) - \beta_1 \int \psi d \mu : \mu \in \mathcal{M}(X, f)\Big\}.
\end{aligned}
\]
Hence, the map \( \beta \mapsto P^{G_\beta}(\Psi) \) is strictly decreasing.  Since \( P^{G_\beta}(\Psi) \) is continuous in \( \beta \), we have that
\[
\inf\left\{\beta \in \mathbb{R} : P^{G_{\beta}}(\Psi) \leqslant 0 \right\}
= \sup\left\{\beta \in \mathbb{R} : P^{G_{\beta}}(\Psi) \geqslant 0 \right\}.
\]
 By Corollary \ref{induce nonliner p cor1}, one has that \( P^{F}_{\psi}(\Phi) \geqslant \inf\left\{\beta \in \mathbb{R} : P^{G_{\beta}}(\Psi) \leqslant 0 \right\} \).

 To complete the proof of the corollary, it suffices to show that
\[
P^{F}_{\psi}(\Phi) \leqslant \inf\left\{\beta \in \mathbb{R} : P^{G_{\beta}}(\Psi) \leqslant 0 \right\}.
\]
By Theorem \ref{induce nonliner p thm} , this is equivalent to show that $R_{\psi, T}^F(\Phi, -\beta \psi, \varepsilon) < \infty$ for sufficiently large $T$ and small $\varepsilon$, provided that $P^{G_{\beta}}(\Psi) < 0 $.

Let \( P^{G_\beta}(\Psi) = 2a < 0 \). For every small \( \varepsilon > 0 \), there exists \( N \in \mathbb{N} \) such that for every positive integer \( n > N \),
\[
 \sup_{E_n} \sum_{x \in E_n} \exp \Big(n F\Big(\frac{ S_n\Phi(x)}{n}\Big)-\beta S_n \psi(x)\Big) \leqslant e^{na},
\]
where the supremum is taken over all $(n, \varepsilon)$-separated sets of $X$.
Consequently,  for sufficiently large \( T>0 \), we have
\begin{align*}
R_{\psi, T}^{F}( \Phi, -\beta \psi, \varepsilon )
& \leqslant \sum_{n \geqslant N} \sup_{E_n} \sum_{x \in E_n} \exp \Big(n F\Big(\frac{ S_n\Phi(x)}{n}\Big)-\beta S_n \psi(x)\Big) \\
& \leqslant \sum_{n \geqslant N} e^{n a}\\
& <\infty.
\end{align*}
This completes the proof of the corollary.
\end{proof}

\begin{remark} Under the condition of Corollary \ref{induce nonliner p cor2},   assume that \( h_{top}(f) < +\infty \), then one has that  \( P^{G_{\beta}}(\Psi) = 0 \) for $\beta= P^{F}_{\psi}(\Phi)$. In fact, $P^{G_{\beta}}(\Psi)\in \mathbb{R}$ for every $\beta\in \mathbb{R}$, and the map \( \beta \mapsto P^{G_\beta}(\Psi) \) is continuous and strictly decreasing. This implies the desired result.
\end{remark}

The following theorem establishes the variational principle for nonlinear induced topological pressure.

\begin{theorem}
Let \((X, f)\) be a TDS, $\Phi \in C(X, \mathbb{R})^d$ and \(\psi \in C(X, \mathbb{R})\) with \(\psi > 0\), and let \(F: \mathbb{R}^{d} \rightarrow \mathbb{R}\) be continuous.
If either of the following two conditions is satisfied:
\begin{enumerate}
\item[(1)] $(f, \Psi)$ has an abundance of ergodic measures, where $\Psi(x) = (\Phi(x), \psi(x))$;
\item[(2)] $F$ is a convex function,
\end{enumerate}
then
\[
P_{\psi}^{F}(\Phi) = \sup \Big\{ \frac{h_{\nu}(f) + F\left(\int \Phi d \nu\right)}{\int \psi d \nu} : \nu \in \mathcal{M}(X, f) \Big\}.
\]
\end{theorem}
\begin{proof}
For every \( \beta\in \mathbb{R}\), let $G_\beta$ be the same as in Corollary \ref{induce nonliner p cor1}. Without loss of generality, one may assume that $P_{\psi}^{F}(\Phi)$ is finite. Otherwise, one can show that $P^{G_{\beta}}(\Psi)=+\infty$ for each $\beta\in \mathbb{R}$, which implies that $h_{\text {top }}(f) = +\infty$. This implies the desired result immediately.

 For every \( \beta > P^{F}_{\psi}(\Phi) \),  by Corollary \ref{induce nonliner p cor2} and Theorem \ref{high nonliner vp},  one has that
\[
0 \ge P^{G_{\beta}}(\Psi) = \sup \Big\{  h_{\nu}(f) + F\Big(\int \Phi d \nu \Big) - \beta \int \psi d \nu : \nu \in \mathcal{M}\left(X, f\right) \Big\}.
\]
Hence,   we have that
$$
0 > \frac{h_{\nu}(f) + F\Big(\int \Phi d\nu\Big)}{\int \psi d\nu} - \beta,\quad \forall \nu \in \mathcal{M}(X, f).
$$
The arbitrariness of \( \beta >P^F_\psi(\Phi) \) and \( \nu \in \mathcal{M}(X, f) \)  implies that
$$
P^F_\psi(\Phi) \geq \sup \Big\{ \frac{h_{\nu}(f) + F\left(\int \Phi d \nu\right)}{\int \psi d \nu} : \nu \in \mathcal{M}(X, f) \Big\}.
$$

For every \( \beta < P^F_\psi(\Phi) \), by Corollary \ref{induce nonliner p cor2} and Theorem \ref{high nonliner vp},  we have that
$$
0 \le P^{G_\beta}(\Psi) = \sup \Big\{ h_\nu(f) + F\Big(\int \Phi d\nu \Big) - \beta \int \psi d\nu : \nu \in \mathcal{M}(X, f) \Big\},
$$
which implies that
$$
\beta \le \sup \Big\{ \frac{h_\nu(f) + F(\int \Phi d\nu)}{\int \psi d\nu} : \nu \in \mathcal{M}(X, f) \Big\}.
$$
Since \( \beta\) is chosen arbitrarily,  one can conclude that
$$
P_\psi^F(\Phi) \leqslant \sup \Big\{ \frac{h_\nu(f) + F(\int \Phi d\nu)}{\int \psi d\nu} : \nu \in \mathcal{M}(X, f) \Big\}.
$$
This completes the proof of the variational principle for nonlinear induced topological pressure.
\end{proof}

	


	

\end{document}